\documentclass[a4paper,11pt]{amsart}
\usepackage{euscript,amsfonts,amss ymb,amsmath,amscd,amsthm,verbatim,ytableau}

\usepackage[top=1in,bottom=1in,left=1.2in,right=1in]{geometry}
\usepackage{graphicx}
\usepackage{color}

\usepackage{hyperref}

\usepackage{multicol}

\usepackage[latin1]{inputenc}

\newtheorem{theorem}{Theorem}[section]
\newtheorem*{theorem*}{Theorem}
\newtheorem{lemma}[theorem]{Lemma}

\newtheorem{proposition}[theorem]{Proposition}

\newtheorem{corollary}[theorem]{Corollary}

\newcommand\F{\mathcal{F}}

\DeclareMathOperator\sgn{sgn}
\DeclareMathOperator\dd{\Delta_s}
\newcommand\bl{\left(}
\newcommand\br{\right)}

\newcommand\GUE{\mathbb{GUE}}

\newcommand\noshow[1]{}

\newcommand{\ii}{\mathbf i}

\newcommand{\X}{\mathfrak{X}}
\newcommand{\la}{\lambda}

\allowdisplaybreaks

\title[Lozenge tilings with free boundaries]{Lozenge tilings with free boundaries }

\author{Greta Panova}

\thanks{University of Pennsylvania, Mathematics Department, Philadelphia, PA 19104. \textsc{panova@upenn.math.edu}. } 



\date{\today }

\begin{document}
\begin{abstract}
We study lozenge tilings of a domain with partially free boundary. In particular, we consider a trapezoidal domain (half hexagon), s.t. the horizontal lozenges on the long side can intersect it anywhere to protrude halfway across.  We show that the positions of the horizontal lozenges near the opposite flat vertical boundary have the same joint distribution as the eigenvalues from a Gaussian Unitary Ensemble (the GUE-corners/minors process). We also prove the existence of a limit shape of the height function, which is also a vertically symmetric plane partition. Both behaviors are shown to coincide with those of the corresponding doubled fixed-boundary hexagonal domain. We also consider domains where the different sides converge to $\infty$ at different rates and recover again the GUE-corners process near the boundary. 
\end{abstract}

\maketitle

\section{Introduction}

We study lozenge tilings of certain domains with partially free boundary conditions. These are tilings with unit-sided rhombi of a domain on a triangular grid, see Figure~\ref{f:domain}. In the particular setting, we are interested in a half-hexagonal domain, depicted in Figure~\ref{f:domain}, such that on its side corresponding to the main diagonal of the hexagon, we have free boundary conditions -- that is,  the ``horizontal'' lozenges are allowed to cross that (free) boundary at any place and protrude halfway through.
Let $m$ be the length of the vertical side of the hexagon, let $n$ be the other two lengths,  and denote by $T^f_{n,m}$ the set of all so--described free--boundary tilings. Reflecting the tiling along the right boundary line of the domain gives a vertically symmetric tiling, which corresponds to a boxed vertically symmetric plane partition fitting in an $m\times n \times n$ box (each lozenge represents a side of a cube). 

Lozenge tilings of fixed boundary domains have been studied extensively both as combinatorial objects corresponding to plane partitions (see e.g. \cite{EC2}) and as integrable models in statistical mechanics, where the interest has been the limiting behavior as the mesh (triangle) size goes to 0 and the domain is fixed in the plane. In these cases, most aspects of the limit behavior have been understood -- the ``frozen regions'' (covered by just one type of tile) bounded by algebraic ``arctic curves''  (see e.g. \cite{KO}, \cite{CLP}), the surface as a limit of the height function, also referred to as ``limit shape'', (see \cite{BBO,BG,CLP,CKP}), the fluctuations near the frozen boundary being the Airy process (see~\cite{Pe} and references therein) and the fact that the positions of the horizontal lozenges near a flat vertical boundary have the same joint distribution as the eigenvalues of Gaussian Unitary Ensemble (GUE) matrices (see \cite{GP,JN,N,Nov,OR}). 
Symmetric lozenge tilings have also been studied when the symmetry is along the $x$-axis instead of the $y$-axis considered here (referring to Figure~\ref{f:domain}), notably in~\cite{FN,BG}. In~\cite{FN} the anti-symmetric GUE minors process is recovered as the limiting behavior of the positions of certain lozenges, similar to the GUE in ordinary tilings. Further, the existence of the limit shape for tilings of more general domains with certain symmetry along the $x$-axis is proven in~\cite{BG}.

\begin{figure}[h]

\includegraphics[height=2.7in]{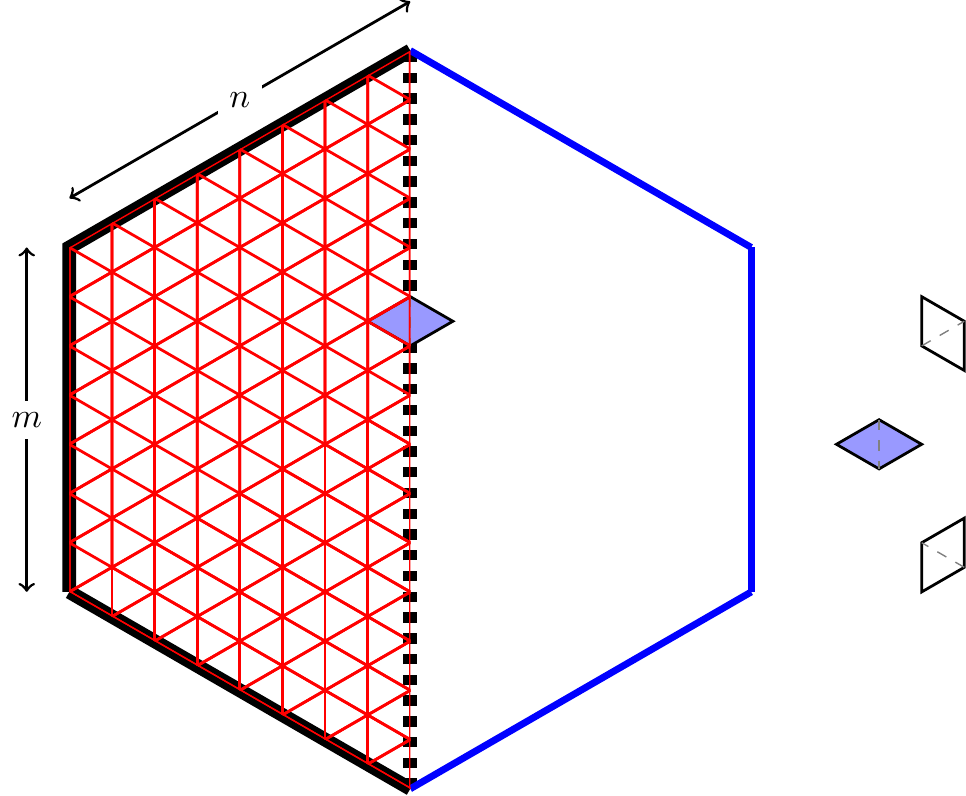}
\caption{The domain in the triangular grid, which is tiled with the 3 types of lozenges (on the right), which is half of a hexagon. The right boundary (middle dashed line) is allowed to cut a lozenge in half as shown and the positions of these ``cut'' lozenges are unrestricted. }\label{f:domain}
\end{figure}

Lozenge tilings with free boundaries, as studied here, do not fit directly into the frameworks used to study tilings with fixed boundaries as listed above. They were studied by Di Francesco and Reshetikhin in \cite{DR} for the same domain as described here.  They computed the arctic curve and the surface corresponding to the limit shape of the height function, assuming (without a rigorous proof) that such a limit exists. Another aspect in free-boundary tilings has been considered in~\cite{CK}, namely, the correlation of a hole punched into the half-hexagon pictured in Figure~\ref{f:domain}. The same domain but with a particular fixed boundary on the diagonal is related to the Aztec diamond (see~\cite{NY}).  Free boundary conditions on another integrable model, the 6-vertex model, are studied in~\cite{BCG}.

In this paper we study the same aspects for lozenge tilings with free boundaries as the ones studied for fixed domains, described above. We show in Section~\ref{s:GUE} that the positions of the horizontal lozenges near the left boundary have the same distribution as the GUE matrices eigenvalues, see Theorem~\ref{t:gue}. We also prove in Section~\ref{s:limit} \emph{the existence of the limit shape} of the height function (i.e. the symmetric plane partition) in Theorem~\ref{t:limit_shape}, its surface is described in~\cite{DR}.  Finally, we consider ``unusual'' scalings of the domain in Section~\ref{s:other}, i.e. such that $m$ and $n$ converge to $\infty$ at different rates, and show that the GUE eigenvalue distribution for the left-most horizontal lozenges still holds as shown in Propositions~\ref{p:n_small} and~\ref{p:n_large}. 

In the ``usual'' scaling regimes our results give the following meta-theorem relating the free boundary case to the full hexagon (as pictured in~Figure~\ref{f:domain}). We let $m/n \to a$ as $n,m \to \infty$.

\begin{theorem}[Theorem~\ref{t:limit_shape}, Corollaries~\ref{c:gue} and~\ref{c:limit}.]
The height function of the uniformly random lozenge tilings of a half-hexagon with free right boundary converges (in probability) to a unique limit shape,  which coincides over the half-hexagon with the limit shape for  the tilings of  the full hexagon (fixed boundary) as described in~\cite{CLP}. Moreover, the shifted by $m/2$ and rescaled by $\sqrt{n (a^2+2a)/8}$  positions of the horizontal lozenges on the first $k$ vertical lines from the left have the same joint distributions as $n,m\to \infty$, which is the distribution of the eigenvalues of the principal $1\times 1, 2\times 2,\ldots, k\times k$ submatrices from the Gaussian Unitary Ensemble, known as the GUE-corners (or GUE-minors) process.
\end{theorem}

\begin{figure}[h]
\includegraphics[height=2.5in]{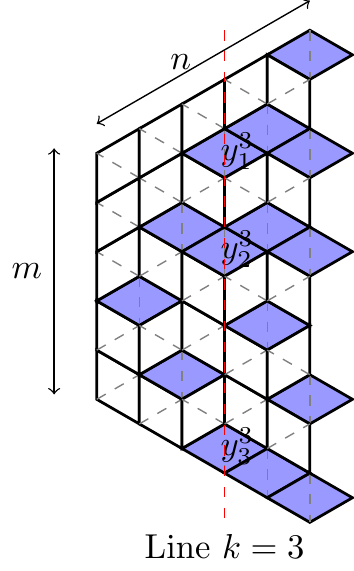} $\qquad \qquad$ \includegraphics[height=2.5in]{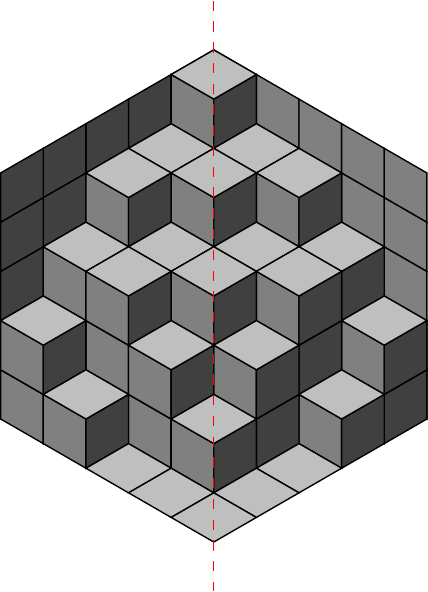}
\caption{Tiling with a free boundary on the right side of the domain, and the corresponding vertically-symmetric boxed plane partition.}\label{f:sym_pp}
\end{figure}

Our methods for proving the GUE distribution follow the approach developed in~\cite{GP} on asymptotics of symmetric functions through certain integral formulas and steepest descent asymptotic analysis. There the GUE-corners process was shown for tilings of fixed boundary domains by using asymptotics of the corresponding moment generating functions, which were certain Schur functions. As it turns out here, the role of Schur functions is replaced by symplectic characters and so we derive their asymptotics explicitly. The proof  of the existence of the limit shape follows the approach of~\cite{BBO} and~\cite{BG} with moment generating functions for certain measures. It relies on some further symmetric function identities and their asymptotics. Finally, the GUE phenomenon for differently scaled $m$ and $n$ is also shown through asymptotic analysis of symplectic characters using the exact formulas from \cite{GP}, with special care needed in the unusual regimes. 

Section~\ref{s:prelim} contains all important background and definitions and, for the sake of self-containment, states the used results from~\cite{GP}. The asymptotic analysis of the symplectic characters in the usual regime and proofs of a general multiplicativity in the multivariate case are in Section~\ref{s:asymp}. Then the main results are stated in their corresponding Sections on GUE, limit shape and the case of different  regimes for the convergence of $m/n$.

\section{Preliminaries}\label{s:prelim}

\subsection{Basic definitions}
For the background on symmetric functions we refer the reader to \cite{Mac} and \cite{EC2}.

Rational irreducible representations of Lie groups are parametrized by their highest weights (also called signatures) $\lambda$ -- 
  sequences of (half)integers $\lambda=(\lambda_1\geq \lambda_2\geq \cdots \geq \lambda_N)$. Denote the set  of highest weights/signatures of length $N$ by $GT_N$.  We will use the characters of these representations. For $GL_N(\mathbb{C})$, these are the Schur functions $s_{\lambda}(x_1,\ldots,x_N)$, which can be defined for any sequence $\lambda$ via Weyl's character formula as
  $$s_{\lambda}(x_1,\ldots,x_N) = \frac{ \det \left[ x_i^{\lambda_j+N-j} \right]_{i,j=1}^{N}}{ \Delta(x)},$$
  where $\Delta(x) = \det \left[ x_i^{N-j} \right]_{i,j=N} = \prod_{i<j} (x_i-x_j)$ is the Vandermonde determinant. When $\lambda \in \mathbb{Z}_{\geq 0}^N$, we have the combinatorial interpretation as sums over semi-standard Young tableaux (SSYT) of shape $\lambda$. 
  Let $[\lambda]$ be the Young diagram of shape $\lambda$, drawn in the plane so that the largest row $\lambda_1$ is on top, the rows are left justified, and entries are indexed $(i,j)$ with $i$ indicating the row number starting from the top, and $j$ is the column number starting from the left. 
  An SSYT of shape $\lambda$ is a map $T:[\lambda] \to \mathbb{N}$,  such that $T_{i,j} \leq T_{i,j+1}$ and $T_{i,j} < T_{i+1,j}$. Then
  $$s_{\lambda}(x) = \sum_{T: sh(T)=\lambda} x^T,$$
  where $x^T = \prod_{(i,j)\in [\lambda] } x_{T_{i,j} }$ -- product over all entries in $T$, where each entry $u$ gives an $x_u$. For a finite number of variables we set $s_{\lambda}(x_1,\ldots,x_N) = s_{\lambda}(x_1,\ldots,x_N,0,\ldots)$. For example when $\lambda=(2,2,0)$ and $N=3$ we have
  \ytableausetup{smalltableaux}
$$s_{(2,2)}(x_1,x_2,x_3)=s_{\ydiagram{2,2}}(x_1,x_2,x_3) =\underset{\ytableaushort{11,22}}{ x_1^2x_2^2 } + 
\underset{\ytableaushort{11,33}}{ x_1^2x_3^2 } + \underset{\ytableaushort{22,33}}{ x_2^2x_3^2 } + \underset{\ytableaushort{11,23}}{ x_1^2x_2x_3 } + \underset{\ytableaushort{12,23}}{ x_1x_2^2x_3 }+\underset{\ytableaushort{12,33}}{ x_1x_2x_3^2 }.$$

The value of the character of the irreducible representation of the symplectic group $Sp_{2N}(\mathbb{C})$, parameterized by
$\lambda\in GT_N$, on a symplectic matrix with eigenvalues $x_1,x_1^{-1},\dots,x_N, x_N^{-1}$ is
given by (see e.g. \cite[p.405]{harris})
$$\chi_{\lambda}(x_1,\ldots,x_N) = \frac{ \det\left[ x_i^{\lambda_j+N+1-j} -
x_i^{-(\lambda_j+N+1-j)}\right]_{i,j=1}^N} {\det\left[x_i^{N+1-j} -
x_i^{-N-1+j}\right]_{i,j=1}^N}.$$ The denominator in the last formula, denoted by $\dd$,  can be simplified as follows
\begin{multline}\label{eq_alternative_form}
\dd(x_1,\ldots,x_N)=\det\left[ x_i^{N-j+1} -
x_i^{-N+j-1} \right]_{i,j=1}^N \\
= \prod_i (x_i-x_i^{-1}) \prod_{i<j}(x_i+x_i^{-1} - (x_j+x_j^{-1}) )
 = \dfrac{\prod_{i<j}(x_i-x_j)(x_ix_j-1)\prod_i (x_i^2-1)}
{(x_1\cdots x_n)^N}.
\end{multline}

\subsection{Normalized characters}

Here we follow the definitions and formulas from \cite{GP}.

The normalized Schur function on $k$ variables is defined via
$$S_{\lambda}(x_1,\ldots,x_k;N) = \frac{ s_{\lambda} (x_1,\ldots,x_k,1^{N-k})}{s_{\lambda}(1^N)}.$$
Similarly, the normalized symplectic character is given by
$$\X_{\lambda}(x_1,\ldots,x_k;N) = \frac{\chi_{\lambda}(x_1,\ldots,x_k,1^{N-k})}{\chi_{\lambda}(1^N)}.$$
Both denominators  admit explicit product formulas. 

The following results were shown in \cite{GP} and are essential here.

\begin{proposition}[Proposition 3.19 in \cite{GP}]
\label{Proposition_Schur_Simplectic_1} For any signature $\lambda\in GT_N$  we have
\begin{equation}
\X_{\lambda}(x;N) = \frac{2}{x+1}S_{\nu}(x;2N),
\end{equation}
where $\nu\in GT_{2N}$ is a signature of size $2N$ given by $\nu_i = \lambda_i+1$ for
$i=1,\ldots,N$ and $\nu_{i}=-\lambda_{2N-i+1}$ for $i=N+1,\ldots, 2N$.
\end{proposition}

The next result gives the asymptotic behavior of the normalized Schur functions when the signature $\lambda(N)$ converges to a limiting profile function $f$ on $[0,1]$ in the sense described below. 
In order to state the precise results here we need to define the following norms measuring the convergence of $\lambda(N)/N$ to its limiting profile $f$. 

$$
 R_1(\lambda(N),f)= \sum_{j=1}^N \left|\frac{\lambda_j(N)}{N}-f(j/N)\right|,\quad
 R_\infty(\lambda(N),f)= \sup_{j=1\dots,N} \left|\frac{\lambda_j(N)}{N}-f(j/N)\right|.
$$
We also introduce $w$ (an inverse Hilbert transform), defined for any $y\in\mathbb C$ by the equation
\begin{equation}
\label{eq_critical_point_equation_formulation} \int_0^1 \frac{dt}{w-(f(t)+1-t) }=y.
\end{equation}
Finally, we define the function $\F(w;f)$
\begin{equation}
\label{eq_definition_F}
 \F(w;f)=\int_0^1 \ln\left(w-(f(t)+1-t) \right)dt, \quad \quad w\in\mathbb C\setminus \left\{f(t)+1-t \mid
 t\in[0,1]\right\}.
\end{equation}
\begin{proposition}[Proposition 4.3 in \cite{GP}]
\label{Prop_convergence_GUE_case} Suppose that  $f(t)$ is piecewise-differentiable,
$R_\infty(\lambda(N),f)=O(1)$ (i.e.\ it is bounded), and $R_1(\lambda(N),f)/\sqrt{N}$ goes to $0$
as $N\to\infty$. Then for any fixed $h\in\mathbb{R}$
$$
 S_{\lambda(N)}(e^{h/\sqrt{N}};N)= \exp\left(\sqrt{N} E(f)h+
 \frac{1}{2} S(f) h^2 + o(1) \right)
$$
as $N\to \infty$, where
$$
 E(f)=\int_{0}^1 f(t) dt,\quad S(f)= \int_0^1 f(t)^2 dt -E(f)^2 + \int_0^1 f(t) (1-2t) dt.
$$
Moreover, the remainder $o(1)$ is uniform over $h$ belonging to compact subsets of $\mathbb
R\setminus \{0\}$.
\end{proposition}

We also use the following, weakest convergence statement.

\begin{proposition}[Proposition 4.1 in \cite{GP}.]
\label{proposition_convergence_mildest} For $y \in \mathbb{R}\setminus \{0\}$, suppose that $f(t)$
is piecewise-continuous, $R_\infty(\lambda(N),f)$ is bounded, $R_1(\lambda(N),f)/N$ tends to zero
as $N\to\infty$, and $w_0=w_0(y)$ is the (unique) real root of
\eqref{eq_critical_point_equation_formulation}. Further, let $y\in \mathbb R\setminus\{ 0\}$ be
such that $w_0$ is outside the interval $[\frac{\lambda_N(N)}{N},\frac{\lambda_1(N)}{N}+1]$ for
all $N$ large enough. Then
\begin{equation}
\label{eq_x37}
 \lim_{N\to\infty}\frac{ \ln S_{\lambda(N)}(e^{y};N)}{N} = y
 w_0-\F(w_0;f)-1-\ln(e^y-1).
\end{equation}
\end{proposition}

\subsection{Multivariate formulas}
The next formulas allow us to derive the asymptotics of normalized characters of $k$ variables, through their asymptotics over a single variable.
Let $D=x\frac{\partial }{\partial x}$ and $D_i = x_i\frac{\partial}{\partial x_i}$. We use the formula for the multivariate symplectic characters from \cite[Theorem 3.17]{GP}, namely 
\begin{multline}\label{multivar_formula}
\X_{\lambda}(x_1,\ldots,x_k;N) = \frac{\Delta_s^1(1^N)}{\Delta_s^1(x_1,\ldots,x_k,1^{N-k})} \times \\
(-1)^{\binom{k}{2}} \det \left[ \left(x_i\frac{\partial}{\partial x_i} \right)^{2(j-1)} \right]_{i,j=1}^k
\prod_{i=1}^k \X_{\lambda}(x_i;N)\frac{(x_i-x_i^{-1}) (2-x_i-x_i^{-1})^{N-1} }{2(2N-1)!},
\end{multline}
where the functions $\Delta_s^1$ are defined for any $k$ and $N\geq k$ in \cite{GP} as 
\begin{multline}
\Delta_s^1(x_1,\ldots,x_k,1^{N-k})=
\dd(x_1,\ldots,x_k) \prod_{i=1}^k\frac{ (x_i-1)^{2(N-k)}}{x_i^{N-k}} \cdot \prod_{1\leq i <j \leq
N-k}(i^2-j^2) 2^{N-k}(N-k)! \; .
\end{multline}
The following formula will be useful in the computations. 
\begin{equation}\label{f:delta_ratios}
\frac{\Delta_s^1(1^N) \Delta_s(x_1,\ldots,x_k) }{\Delta_s^1(x_1,\ldots,x_k,1^{N-k})} 
=2^k (-1)^{Nk-\binom{k+1}{2}} \prod_{j=N-k+1}^N (2j-1)!  \cdot \prod_{i=1}^k \frac{ x_i^{N-k}}{  (x_i-1)^{2N-2k} }.
\end{equation}

\subsection{Generating functions}

We use the following crucial formula from \cite[I.5, Example 16]{Mac}, which we refer to as \textbf{Macdonald's identity}, to compute the sum of Schur functions, indexed by partitions $\lambda \subset (m^n)$, where $(m^n)=(\underbrace{m,\ldots,m}_n)$ is the rectangular partition:

\begin{equation}\label{phi_sum}
\phi_m(x_1,\ldots,x_n):= \sum_{\lambda \in (m^n)} s_{\lambda}(x_1,\ldots,x_n) = \frac{ \det [ x_j^{m+2n-i} -x_j^{i-1} ]_{1\leq i,j \leq n} }{\det [ x_j^{2n-i} -x_j^{i-1} ]_{1\leq i,j \leq n}}.
\end{equation}
The right-hand side of \eqref{phi_sum} is also (a shifted version of) Weyl's dimensional formula for the character $\gamma_{\la}$ corresponding to the irreducible representation of highest weight $\la$ of the odd orthogonal group $O_{2n+1}(\mathbb{C})$ (see \cite[\S24.2]{harris}), given in general by
$$\gamma_{\lambda}(x) = \frac{ \det[ x_j^{\lambda_i+n-i+\frac12}- x_j^{-(\lambda_i+n-i+\frac12)}]_{i,j=1}^n }{\det[ x_j^{n-i+\frac12}- x_j^{-(n-i+\frac12)}]_{i,j=1}^n }.$$
 Using a combinatorial interpretation by Seshadri coming from algebraic geometry, Macdonald's identity~\eqref{phi_sum} is evident in Proctor's work (see e.g. \cite{BLPP})  and further generalized by Krattenthaler in \cite{Kr}. It relies on the ``branching rule'' for the restriction of a representation of $o(2n+1)$ to a subalgebra $sl(n)$. However, no direct combinatorial proof is known relating the semi-standard Young tableaux of $n$ letters which fit into the $(m^n)$ rectangle, and the combinatorial interpretation of the orthogonal characters, coming from the $SO_{2n-1} \hookrightarrow SO_{2n+1}$ branching rule, and corresponding to certain ``half'' Gelfand-Tsetlin triangles (see \cite{Zh}).

While  the orthogonal characters can be analyzed similarly using the general approach in \cite{GP}, here we take advantage of some already developed asymptotics for the normalized Schur functions and their relationship to normalized symplectic characters, see Proposition~\ref{Proposition_Schur_Simplectic_1}. We rewrite the right-hand side of formula~\eqref{phi_sum} as a ratio of symplectic characters, as follows.

We have the following relationship between orthogonal and symplectic characters for any signature $\lambda$. 
\begin{equation}\label{orth_symp} \gamma_{\lambda}(x_1,\ldots,x_n) = \frac{ \det[ x_j^{\lambda_i-\frac12+n-i+1}- x_j^{-(\lambda_i-\frac12+n-i+1)}]_{i,j=1}^n }{\det[ x_j^{-\frac12+ n-i+1}- x_j^{-(-\frac12 + n-i+1)}]_{i,j=1}^n } = \frac{ \chi_{\lambda-(\frac12)^n} (x_1,\ldots,x_n)}{\chi_{(-\frac12)^n}(x_1,\ldots,x_n)}.\end{equation}

In order to express $\phi_m$ in terms of a specific symplectic character, we note that

$$\phi_m(x_1,\ldots,x_n) = \prod_i x_i^{m/2} \cdot \gamma_{\left(\frac{m}{2}\right)^n} (x_1,\ldots,x_n).$$

Denote $\tau^r = \bl (r/2-1/2)^n \br$ and rewrite~\eqref{phi_sum} using the relationship~\eqref{orth_symp} as

$$\phi_m(x_1,\ldots,x_n) =
\prod_i x_i^{\frac{m}{2}}  \frac{\chi_{\tau^m}(x_1,\ldots,x_n)}{ \chi_{\tau^0}(x_1,\ldots,x_n)}.
$$ Note that all necessary formulas for the symplectic characters apply for non-integral partitions as well, so the fractions in $\tau^m$ and $\tau^0$ do not impose any issues.

We consider the normalized version of $\phi_m(x)$, namely
\begin{equation}\label{diamond}
 \Phi_m(x_1,\ldots,x_k;n) = \frac{\phi_m(x_1,\ldots,x_k,1^{n-k} )}{\phi_m(1^n)} = \prod_{i=1}^k x_i^{\frac{m}{2}} \frac{\X_{\tau^m}(x_1,\ldots,x_k;n)}{\X_{\tau^0}(x_1,\ldots,x_k;n) }.
 \end{equation}

This is exactly the \textbf{moment generating function} for the distribution over tilings with a free boundary, as we see later.

\section{Asymptotics I: when $\lim_{n\to \infty} \frac{m}{n} \in (0,+\infty)$ }\label{s:asymp}

Here we develop the general asymptotic results for symplectic characters, similar in nature to the results about Schur functions presented in \cite[Section 3,4]{GP}. 

\subsection{Asymptotics for univariate symplectic characters.}

We derive first the asymptotic behavior for the normalized symplectic characters when $k=1$.

\begin{proposition}\label{prop:univar}
When $\frac{m}{n} \to a$ for $0<a <\infty$, with $\sqrt{n}(m/n-a) \to 0$, we have that
$$\X_{\tau^m}(e^{ \frac{h}{\sqrt{n}} };n) = \exp\left( \frac{1}{16} (a^2+2a) h^2 +o(1) \right)$$
and
$$\X_{\tau^0}(e^{h/\sqrt{n}};n) = 1+o(1),$$
where the error terms $o(1)$ converge to 0 uniformly on compact real domains for $h$. 
\end{proposition}

\begin{proof}
We first derive the asymptotics for the normalized Schur functions for the corresponding signatures, in view of Proposition~\ref{Proposition_Schur_Simplectic_1}
$$S_{\nu^m}(e^{ \frac{h}{\sqrt{2n}} };2n) = \exp\left( \frac{1}{32} (a^2+2a)h^2 + o(1)\right),$$
where $\nu^m =  ( (\frac{m}{2}+\frac12)^n, (-\frac{m}{2}+\frac12)^n))$.
 We apply the asymptotic formula~\ref{Prop_convergence_GUE_case} directly with $\lambda(2n) =\nu^m$. Compute $f(t) = \frac{a}{4}$ for $t\in [0,\frac12]$ and   $f(t)=-\frac{a}{4} $ for $t\in(\frac12,1]$.  Note that $N=2n$ in this case. We also have that 
 $$\frac{R_1(\nu^m,f)}{\sqrt{2n}}= \frac{1}{2\sqrt{2}} \sqrt{n}|m/n-a| + o(1/\sqrt{n}),$$ and the formula applies as long as $\sqrt{n}(m/n-a) \to 0$.   We then obtain
\begin{align*}
E(f) = 0, \quad S(f) = \left(\frac{a}{4} \right)^2 + \frac{a}{4} \int_0^{\frac12} (1-2t)dt - \frac{a}{4} \int_{\frac12}^{1}(1-2t)dt \\
= a^2/16 +a/8 = (a^2+2a)/16
 \end{align*}
 
 Together with Proposition~\ref{Proposition_Schur_Simplectic_1}, this gives that
 $$\X_{\tau^m}(e^{h/\sqrt{2n} };n) = \frac{2}{e^{h/\sqrt{2n}}+1} \exp( \frac{1}{16}(a^2+2a)h^2 +o(1) ) ,$$
when $m/n \to a$, s.t. $0<a< \infty$.

We will investigate the case of $m/n \to 0$ (which includes $m=0$) and $m/n \to \infty$ separately in Section~\ref{s:other}.

Consider now the case when $m= 0$, for which we recall the  integral formula from \cite[Section 3.2]{GP}: 
\begin{equation}\label{f:int}
S_\lambda(x;N) = \frac{(N-1)!}{(x-1)^{N-1} }\frac{1}{2\pi \ii} \oint_C
\frac{x^z}{\prod_{i=1}^N(z-(\lambda_i+N-i))}dz.
\end{equation}
Here the contour $C$ contains all finite poles of the integrand and the formula holds for all $x \neq 0,1$. 

Setting $\widehat{\lambda}_i = \beta \lambda_i + (\beta-1)(N-i)$ for any positive real number $\beta$, we get that
\begin{align*}
S_\lambda(x;N) = \beta^{N} \frac{(N-1)!}{(x-1)^{N-1} }\frac{1}{2\pi \ii} \oint_C
\frac{x^z}{\prod_{i=1}^N(\beta z-\beta (\lambda_i+N-i))}dz\\=
\beta^{N-1} \frac{(N-1)!}{(x-1)^{N-1} }\frac{1}{2\pi \ii} \oint_C
\frac{(x^{1/\beta})^{\beta z}}{\prod_{i=1}^N(\beta z-\beta (\lambda_i+N-i))}d(\beta z)\\
=\beta^{N-1} \frac{(N-1)!}{(x-1)^{N-1} }\frac{1}{2\pi \ii} \oint_C
\frac{(x^{1/\beta})^{\widehat{z}}}{\prod_{i=1}^N(\widehat{z}-(\beta \lambda_i+(\beta-1)(N-i) + N-i))}d\widehat{z}\\
= (\beta\frac{x^{1/\beta}-1}{x-1})^{N-1}S_{\widehat{\lambda}}(x^{1/\beta};N) ,
\end{align*}

This formula now allows us to study the asymptotics of $S_{\nu^0}$.

We have that, with $\beta=2$,  $\widehat{\nu^0}_i = (1+N-i)$ for $i=1,\ldots,N$, where $N=2n$. 
Thus, for the limiting profile of this $\widehat{\nu^0}$ we have $f(t) = 1-t$, so that $E(f) = \frac12$ and $S(f) = \frac12-1/4=1/4$ and then Proposition~\ref{Prop_convergence_GUE_case} gives

$$S_{\widehat{\nu^0}}(e^{\frac{h}{\sqrt{N}} };2N) = \exp( \sqrt{N} \frac12 h + 1/8h^2 + o(1) ).$$ 

Multiplying it with the expansion $$ \left(2\frac{e^{h/(\sqrt{N})}-1}{e^{2\frac{h}{\sqrt{N}}}-1}\right)^{N-1} = \exp( -\frac{h}{2}\sqrt{N} - h^2/8 +o(1) ),$$ we obtain for any $y$ that
$$S_{\nu^0}(e^{\frac{2h}{\sqrt{N}} }; N) =  \exp( 0+ o(1) ),$$
which gives directly the desired result for $\X_{\tau^0}$ as well. 
We note that the uniform convergence of the error term is a direct consequence of the applied asymptotic formulas from Proposition~\ref{Prop_convergence_GUE_case}.
\end{proof}

As a corollary of this statement formula~\eqref{diamond} gives

$$\Phi_m(e^{h/\sqrt{n}};n) = \exp\left( \frac{a}{2}\sqrt{n} h   + \frac{ a^2+2a}{16}h^2 + o(1)\right )$$
 
 for $m/n \to a$ with  $\sqrt{n}(m/n - a) \to 0$ as $n \to \infty$. 


\subsection{Asymptotics in the multivariate case.}

We derive the following statement, in a more general form than the similar statement for normalized Schur functions stated in~\cite{GP}. Note that this general form can be also easily derived for the Schur functions as well.

\begin{proposition}\label{prop:multivar}
Let $\{a_i\}_{i=1}^{\infty}, \{b_i\}_{i=1}^{\infty}$ be sequences of positive real numbers, such that 
$$a_N \to 0 \;, \quad \text{ and } \quad \frac{b_N}{N} \to 0 ,$$
such that either $b_N \to \infty$ as $N \to \infty$
or $b_N = 0$ for all $N$. 
Suppose we have that for some number $b$
$$\X_{\lambda(N)}\left( e^{a_N y};N\right) e^{b_N y} \to g(y) $$
uniformly on compact subsets of a domain $A \subset \mathbb{C}$ as $N\to \infty$. Then
$$\lim_{N\to \infty} \X_{\lambda(N)}\left(e^{y_1a_N},\ldots,e^{y_ka_N} ;N\right) e^{b_N(y_1+\cdots +y_k)} = g(y_1)\cdots g(y_k) $$
uniformly on compact subsets of $A^k$.
\end{proposition}

{\bf Remark 1.} This result will be used primarily for $a_n = 1/\sqrt{n}$ and $b_n=0$ to analyze the convergence of $\Phi_m$ in the case when $m/n$ is bounded nonzero, and later in Section~\ref{s:other} with $a_n = \frac{\sqrt{n}}{m}$ and $b_n= 0$ when $m/n \to \infty$.

\begin{proof}

Let $$g_N(y)=\X_{\lambda(N)}\left( e^{a_Ny};N\right) e^{b_Ny}.$$ Denote for brevity $$p_N(x) = (x-x^{-1})(1-x)^{N-1}(1-x^{-1})^{N-1}.$$
Notice first that $$x\frac{\partial}{\partial x} f(x) \big |_{x=e^{a_Ny} } = a_N^{-1} \frac{\partial}{\partial y} f(e^{ya_N} )$$ for any function $f$. Thus, when $b_N \to \infty$ we have

\begin{equation}\label{DX}
D^\ell  \left( g_N(y) e^{-b_Ny}
\right)= a_N^{-\ell} \bl \frac{\partial}{\partial y} \br^\ell (g_N(y) e^{-yb_N }) = (-b_Na_N^{-1})^{\ell} \left[ e^{-yb_N }g_N(y)(1+ O\left(b_N^{-1}\right)) \right]
\end{equation}
since the derivatives of $g_N(y)$ are uniformly bounded as well on compact subsets. The convergence of derivatives holds in the current because the functions are actually Laurent polynomials in exponents of $e^y$, hence analytic.  
When $b_N \equiv 0$ for all $N$, we just have 
\begin{equation*}
D^\ell  \left( g_N(y) 
\right)= a_N^{-\ell} \bl \frac{\partial}{\partial y} \br^\ell g_N(y) 
\end{equation*}
Fix $N$ and set $x=e^{a_N y}$, $x_i=e^{y_ia_N}$ for all $i$, throughout the computations below we use both  $x_i$ and $y_i$ in the formulas, depending on which is more convenient for the current purpose.

We have
\begin{equation}\label{eq:Dp}
\frac{ D p_N(x)}{p_N(x)} = N \frac{x+1}{x-1} + \frac{-2x}{x^2-1}.
\end{equation}
Keeping in mind that $x=e^{a_Ny} =1 + O(a_N)$, we have by induction on $\ell$ that
\begin{equation}\label{Dp} \frac{D^\ell p_N(x)}{N^\ell (x+1)^\ell(x-1)^{-\ell} p_N(x)} =1 + O(1/N),\end{equation}
on any compact domains of $x$ in $(0,+\infty)$. In particular, we have that 
$$\frac{ D^\ell p_N(x)}{p_N(x)} \Bigg|_{x=e^{ya_N}} = N^\ell2^\ell (ya_N)^{-\ell} \left(1+O(a_N) +O(1/N) \right)$$

Combining \eqref{DX} with the later formulas we get
\begin{multline*}
D^j (\X_{\lambda(N)}(e^{ya_N};N)p_N(e^{ya_N})) = \sum_{\ell=0}^j \binom{j}{\ell}(D^\ell  (g_N(y) e^{-b_N y})) D^{j-\ell} p_N(e^{ya_N})\\
=\sum_{\ell=0}^j \binom{j}{\ell} g_N(y) y^{-j+\ell}e^{-b_Ny}p_N(e^{ya_N})\frac{(2N)^j}{a_N^j} \left( -\frac{b_N}{2N} \right)^\ell (1+O(a_N+\frac{1}{N})).
\end{multline*}
Since $\frac{b_N}{N} \to 0$ as $N\to \infty$, the largest order term above is the one for $\ell=0$, so
\begin{equation}\label{Xp}
 D^j [\X_{\lambda(N)}(e^{ya_N})p_N(e^{ya_N})] = g_N(y) e^{-b_Ny} D^j(p_N(e^{ya_N}))(1+O(a_N+N^{-1}+b_N N^{-1})). 
\end{equation}
When $b_N \equiv 0$, the same formula holds again, since 
\begin{multline*}
D^j (\X_{\lambda(N)}(e^{ya_N};N)p_N(e^{ya_N})) = \sum_{\ell=0}^j \binom{j}{\ell}(D^\ell  (g_N(y))) D^{j-\ell} p_N(e^{ya_N})\\
=\sum_{\ell=0}^j \binom{j}{\ell} a_N^{-j} \frac{\partial^\ell}{\partial y^\ell} g_N(y) (2N)^{j-\ell} y^{-j+\ell}p_N(e^{ya_N}) \left( 1+O(a_N+N^{-1})\right)
\end{multline*}
and the highest order term comes again from $\ell=0$. 

The multivariate formula \eqref{multivar_formula} gives
\begin{multline*}
\X_{\lambda(N)}(x_1,\ldots,x_k;N) 
=\frac{\Delta_s^1(1^N)}{\Delta_s^1(x_1,\ldots,x_k,1^{N-k})} \times \\
(-1)^{\binom{k}{2}+(N-1)k} (2(2N-1)!)^{-k}\det \left[ D_i^{2(j-1)}\X_{\lambda(N)}(x_i;N)p_N(x_i) \right]_{i,j=1}^k.
\end{multline*}
Substituting the formulas \eqref{Xp} inside the determinants we have

\begin{multline}\label{f:multivar2}
\X_{\lambda(N)}(e^{y_1a_N} ,\ldots,e^{y_ka_N };N) = (-1)^{\binom{k}{2}+(N-1)k} (2(2N-1)!)^{-k} \frac{\Delta_s^1(1^N)}{\Delta_s^1(x_1,\ldots,x_k,1^{N-k})} \Bigg|_{x_i=e^{y_ia_N}}\times\\
\prod_{i=1}^kg_N(y_i) e^{-b_Ny_i} \det\left[ (D_i^{2j-2} p_N(x_i))\big|_{x_i = e^{y_ia_N} }\right]_{i,j=1}^k (1+O(a_N+N^{-1}+b_N N^{-1})).
\end{multline}

We determine the asymptotics of the determinant in the above formula using the asymptotics for $p_N$ from~\eqref{Dp}:
\begin{multline*}
\det\left[ (D_i^{2j-2} p_N(x_i))\big|_{x_i = e^{y_ia_N} }\right]_{i,j=1}^k =
\det \left[ N^{2j-2} \left(\frac{ x_i+1}{x_i-1} \right)^{2j-2} p_N(x_i) (1+O(1/N)) \right]_{i,j=1}^k \Big|_{x_i = e^{y_ia_N} }
\\
=
N^{k(k-1)} \prod_{i=1}^k p_N(x_i) \cdot \det\left[ \left( \frac{x_i+1}{x_i-1}\right)^{2(j-1)} \right]_{i,j=1}^k (1+O(N^{-1}))\Bigg|_{x_i=e^{y_ia_N} }\\
=
N^{k(k-1)} \prod_{i=1}^k p_N(x_i) \cdot \prod_{i>j} \left( \left(\frac{x_i+1}{1-x_i}\right)^2  - \left(\frac{x_j+1}{1-x_j}\right)^2\right) (1+O(N^{-1}))\Bigg|_{x_i=e^{y_ia_N} } \\
=
N^{k(k-1)}(-1)^{(N-1)k} 2^{k(k-1)} \Delta_s(x_1,\ldots,x_k)
 \prod_{i=1}^k \frac{(1-x_i)^{2N-2k}}{x_i^{N-k}} 
(1+O(N^{-1}))\Bigg|_{x_i=e^{y_ia_N} } 
\end{multline*}
Substituting the last result in the formula for the multivariate normalized character \eqref{f:multivar2} and using formula \eqref{f:delta_ratios}, we get

\begin{multline*}
\frac{ \X(e^{y_1a_N} ,\ldots,e^{y_ka_N};N)  }{ \prod_{i=1}^kg_N(y_i) e^{-b_Ny_i} }=
(-1)^{\binom{k}{2}+(N-k)k} \frac{(2N)^{k(k-1)}   }{(2(2N-1)!)^k} \times\\ \frac{\Delta_s^1(1^N)\Delta_s(x_1,\ldots,x_k)}{\Delta_s^1(x_1,\ldots,x_k,1^{N-k})} 
\cdot \prod_{i=1}^k \frac{(1-x_i)^{2N-2k}}{x_i^{N-k}}  \cdot (1+O(*))\\
=(-1)^{\binom{k}{2}+(N-k)k} \frac{(2N)^{k(k-1)}   }{(2(2N-1)!)^k} \cdot 2^k(-1)^{Nk-\binom{k+1}{2}} \prod_{j=N-k+1}^N (2j-1)! \cdot   (1+O\left(*\right))\\
= \prod_{j=N-k+1}^N \frac{ (2N)^{2N-2j} (2j-1)!}{(2N-1)!} \cdot   (1+O(*)) =1 +O\left(a_N+N^{-1}+b_NN^{-1}\right), \\
\end{multline*}
where $O(*)$ is the error term, equal to $O(a_n+N^{-1}+b_N N^{-1} )$. 
Accounting for the fact that in all asymptotic approximations above the convergence is uniform on compact domains of the $y_i$s, we get the desired
formula.
\end{proof}

\begin{proposition}\label{p:ln_mult}
Suppose that $\lambda(N)$ is a sequence of signatures, such that 
$$\frac{\ln\X_{\la(N)}(u;N)}{N} \to X(u)\; , \qquad \text{as $N\to \infty$},$$
where the convergence is uniform in a complex neighborhood of $u=1$. 
Then we have 
$$\frac{\ln\X_{\la(N)}(u_1,\ldots,u_k;N)}{N} \to X(u_1)+\cdots + X(u_k)$$
as $N\to \infty$, where the convergence is uniform in a neighborhood of $(1^k)$ in $\mathbb{C}^k$. 
\end{proposition}
\begin{proof}
The proof is completely analogous to the proof of \cite[Corollary 3.11]{GP}. 
Observe that $\ln \X_{\lambda(N)}(u;N)$ are analytic functions in a compact complex neighborhood of $u=1$ and thus by Weierstrass convergence theorem $X(u)$ is also analytic and the corresponding derivatives converge uniformly. The analyticity follows, for example, from the fact that the functions   $\X_{\lambda(N)}(u)$ are Laurent polynomials in $u$, which follows from their combinatorial interpretations as generating functions over symplectic tableaux (similarly to the Schur functions and SSYTs, see e.g.~\cite{Sagan}). So $\X_{\lambda(N)}(u)$ is analytic in any compact complex neighborhood of $u=1$ with $Re(u)>0$. 
Note also that 
$$\frac{ \left( \frac{\partial}{\partial u} \right)^j \X_{\lambda(N)}(u;N)}{\X_{\lambda(N)}(u;N)} \in \mathbb{Z}\left[ \frac{\partial}{\partial u}\ln \X_{\lambda(N)}(u;N),\ldots,\frac{\partial^j}{\partial u^j} \ln \X_{\lambda(N)}(u;N) \right],$$
a degree $j$ polynomial in the partial derivatives of $\ln\X_{\lambda(N)}(u;N)$.  Then we have that $(D_u^j \X_{\lambda(N)}(u;N))/\X_{\lambda(N)}(u;N)$ is a polynomial in the above ring (adding $u$ as a variable). 
Hence $$ \frac{D_u^j \X_{\lambda(N)}(u;N)}{N^j \X_{\lambda(N)}(u;N)}$$ converge uniformly as well: note that these terms are bounded, because the functions in the corresponding ring are of order $N$.

Now we use the multivariate formula~\eqref{multivar_formula} and as in the proof of Proposition~\ref{prop:multivar}, we have that the $(i,j)$ entries in the determinant are of the form
\begin{equation}\label{bounded_det} 
N^{2j-2} \times \X_{\lambda(N)}(x_i;N)p_N(x_i) \cdot \sum_{\ell=0}^{2j-2} \binom{2j-2}{\ell}\frac{D_i^{\ell}\X_{\lambda(N)}(x_i;N)}{N^\ell\X_{\lambda(N)}(x_i;N)} \frac{D_i^{2j-2-\ell}p_N(x_i)}{N^{2j-2-\ell} p_N(x_i)}
\end{equation}
By equation~\eqref{Dp}, we have that $\frac{D_i^{2j-2-\ell}p_N(x_i)}{N^{2j-2-\ell}p_N(x_i)}$ is bounded. 
We can now take the factors $N^{2j-2} X_{\lambda(N)}(x_i;N) p_N(x_i)$ from the determinant in~\eqref{multivar_formula}. Multiplying them with the prefactor there and using its form in equation~\eqref{f:delta_ratios}, we obtain the following  bounded prefactor
\begin{multline*}
 \frac{\Delta_s^1(1^N)}{\Delta_s^1(x_1,\ldots,x_k,1^{N-k})} \prod_{i=1}^k N^{2j-2} p_N(x_i) \\
 =
 \Delta_s(x_1,\ldots,x_k)^{-1} (-1)^{\binom{k+1}{2}} \prod_{i=1}^k \frac{(x_i+1)(x_i-1)^{2k-1}}{x_i^k}  \prod_{j=N-k+1}^N \frac{N^{2N-2j}}{(2j)(2j+1)\cdots (2N-1)}.
\end{multline*}   
Combining this prefactor with the determinant of the bounded terms from~\eqref{bounded_det},  we obtain the ratio
$$\frac{\X_{\lambda}(x_1,\ldots,x_k;N)}{\prod \X_{\lambda}(x_i;N)}$$
as a bounded function of the $x$s as $N \to \infty$, as long as $x_i \neq x_j$. By analyticity of the functions around $x=(1^k)$, we must have that the ratio is then bounded for all $x$ in a compact neighborhood.  As the value at $x=(1^k)$ is actually 1, its logarithm is also bounded, and dividing it by $N$ we obtain the desired limit and uniform convergence.
\end{proof}

\subsection{Asymptotics of the generating function $\Phi_m$}

We apply the asymptotic results on the normalized symplectic characters $\X_{\la}(x_1,\ldots,x_k;n)$ to obtain the asymptotic results on the desired generating function $\Phi_m$.

Let again $\frac{m}{n}\to a$, then Proposition \ref{prop:univar} gives that 
$$\X_{\tau^m}(e^{h/\sqrt{n}};n) = \exp\left(\frac{1}{16}(a^2+2a)h^2 +o(1)\right),$$ which satisfies the assumptions of Proposition \ref{prop:multivar} with $g(y) =  \exp\left(\frac{1}{16}(a^2+2a)y^2\right)$ and $b_n=0$, so 
$$\X_{\tau^m}(e^{h_1/\sqrt{n}},\ldots,e^{h_k/\sqrt{n}};n) =  \exp\left( \frac{1}{16}(a^2+2a)\sum_{i=1}^k h_i^2\right)(1+o(1)).$$

We further have in the case of $m=0$ that 

$$\X_{\tau^0}(e^{h/\sqrt{n}};n) =  \exp\left( 0+o(1) \right),$$
so again applying Proposition \ref{prop:multivar} with $g(y) = 1$, $a_n=1/\sqrt{n}$ and $b_n=0$ we get

$$\X_{\tau^0}(e^{h_1/\sqrt{n}},\ldots,e^{h_k/\sqrt{n}};n) = 1+o(1).$$
The above asymptotic formulas for $\tau^m$ and $\tau^0$, and equation \eqref{diamond} then give us

\begin{theorem}\label{t:Phi_m_convergence}
For any fixed $y_1, y_2, \ldots, y_k \in \mathbb{R}$ we have that
$$\Phi_m(e^{y_1/\sqrt{n}},\ldots,e^{y_k/\sqrt{n}};n) \exp\left(- \frac{m}{2\sqrt{n}}\sum_{i=1}^k y_i \right)= \exp\left(  \frac{1}{16}(a^2+2a)\sum_i  y_i^2 +o(1)\right),$$
where the convergence of the $o(1)$ term as $n \to \infty$, $\sqrt{n}(m/n- a) \to 0$,  is uniform over $(y_1,\ldots,y_k)$ belonging to compact domains of $\mathbb{R}^k$. 
\end{theorem}

\section{Convergence to eigenvalues of GUE corners}\label{s:GUE}
We now consider the joint distribution of the positions of the horizontal lozenges closest to the left flat boundary of the domain. Let $Y^k_{n,m}$ denote the $k-$tuple 
$$Y^k_{n,m}=\bl y^k_1+k-1, y^k_2+k-2,\ldots, y^k_k\br,$$
 where $y^k=\{y^k_j\}$ are the positions of the horizontal lozenges on the $k$-th from the left vertical line a uniformly random tiling $T^f_{n,m}$, as depicted in Figure~\ref{f:y_k}. Here the position of the horizontal lozenge is specified by the vertical distance along the vertical line from its lower corner to the bottom side of the trapezoidal domain.  We are going to show that the joint distribution of the ensemble $[Y^1_{n,m}, \ldots, Y^k_{n,m} ]$, shifted by $m/2$ and then rescaled by a factor of $\sqrt{n}$, converges to the distribution of the eigenvalues of the top principal submatrices of sizes $1,\ldots,k$ of a matrix from the Gaussian Unitary Ensemble (known as GUE corners or GUE-minors) described below.  

\begin{figure}[h]
\center
\includegraphics[width=0.5\textwidth]{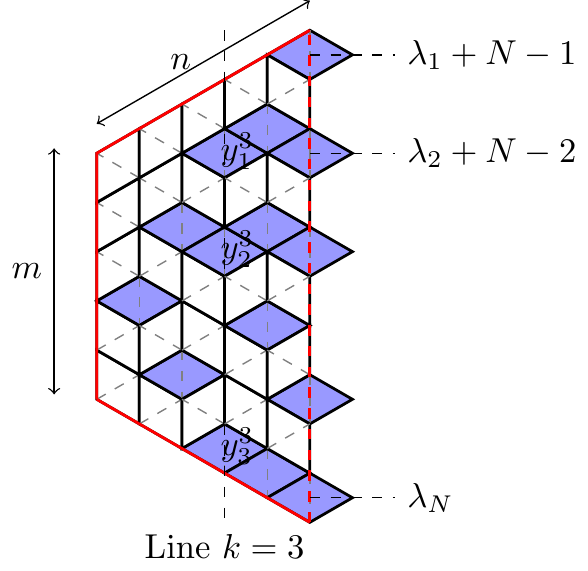}

\caption{The positions $y^k$ of the horizontal lozenges on vertical line $k$ of a tiling whose rightmost boundary is given by the partition $\lambda=(5,4,3,1,0)$. Here $k=3$ and $y^3= (4,3,0)$.} \label{f:y_k}
\end{figure}
In order to prove the convergence to GUE we use moment generating functions. Introduce the multivariate normalized Bessel function in $(x;y)=(x_1,\ldots,x_k;y_1,\ldots,y_k)$, defined as
$$B_k(x;y) = \frac{ \det \left[ \exp(x_iy_j) \right]_{i,j=1}^k }{\prod_{i<j} (x_i-x_j) \prod_{i<j}(y_i-y_j)}\prod_{i<j}(j-i).$$
Let $\GUE_k$ denote  the $k$ eigenvalues $\varepsilon_1\geq \varepsilon_2 \geq \cdots \geq \varepsilon_k$ of a random Hermitian $k\times k$ matrix from a Gaussian Unitary Ensemble and let $P_{\mathbb{\GUE}_k}$ denote their joint distribution. It is well known (see e.g. \cite[Theorem 3.3.1]{Mehta}), that
\begin{equation}\label{P_GUE}
P_{\GUE_k}(\varepsilon_1,\ldots,\varepsilon_k) = \frac{1}{Z_{2,k}} \Delta(\varepsilon)^2 \exp(-( \varepsilon_1^2+\cdots +\varepsilon_k^2) ),
\end{equation}
where $Z_{2,k} = (2\pi)^{k/2} 2^{-k^2/2} \prod_{j=1}^{k} j!$ is the normalization constant (see e.g. \cite{Mehta}). 
The following was shown in \cite[Prop 5.6]{GP} for the expectation of $B_k(x;y)$ where $y$ is distributed as the eigenvalues of the $k\times $k matrix from GUE ensemble.
\begin{proposition}[Prop. 5.6 in \cite{GP}]\label{l:gue}
We have that $$\mathbb{E}B_k(x;\mathbb{GUE}_k) = \exp\left( \frac{1}{2}( x_1^2+\cdots x_k^2)\right).$$
\end{proposition}
A self-contained proof of this fact, based only on equation~\eqref{P_GUE}, is given at the end of this Section.

The following result is the key component in the path to showing that the moment generating function for the shifted and rescaled $Y^k_{n,m}$ converges to the corresponding m.g.f. for $\GUE_k$.

\begin{proposition}\label{lemma:EB=Phi}
The expectation of $B_k(x;Y^k_{n,m})$, where $Y^k_{n,m}$ is the $k$-tuple of the positions of the horizontal lozenges at line $k$ (from the left) of a tiling  chosen uniformly from $T^f_{n,m}$,  is given by
$$\mathbb{E}B_k\bl x;\frac{Y^k_{n,m}-m/2}{\sqrt{n} }\br = \prod_{i=1}^k \exp(-\frac{m }{2\sqrt{n}}x_i)  \Phi_m(e^{x_1/\sqrt{n}},\ldots,e^{x_k/\sqrt{n}};n).$$ 
\end{proposition}

\begin{proof}
First, we calculate the probability that $\{y^k_j\}$ are the positions of the horizontal lozenges on the $k$-th vertical line of a uniformly random tiling from f$T^f_{n,m}$. Tilings in $T^f_{n,m}$ are in bijection with Gelfand-Tsetlin patterns of $n$ rows and entries no larger than $m$. Those, on the other hand, are in bijection with semi-standard Young tableaux (SSYT) $T$ filled with the numbers $1,\ldots,n$ and shapes $\lambda$ (corresponding to the positions of the right-most horizontal lozenges as in Figure~\ref{f:y_k}) that fit in the $m^n$ rectangle, i.e. $\lambda_1\leq m$. The positions of the horizontal lozenges on the $k$-th vertical line are simply the entries on the $k$-th row of the Gelfand-Tsetlin pattern. These correspond to the sub-tableaux of $T$ occupied by the entries which are no larger than $k$.  See Figure~\ref{f:GT-SSYT}.

\begin{figure}[h]
\center
\begin{tabular}{p{2in}p{0.5in}p{2in}}
$\begin{array}{ccccccccc}
&&&& \textcolor{blue}{2}&&&& \\
&&&\textcolor{blue}{0} &&\textcolor{blue}{3}&&&\\
&&\textcolor{blue}{0}&&\textcolor{blue}{1}&&\textcolor{blue}{3}&&\\
&0&&0&&3&&3&\\
0&&0&&3&&3&&4
\end{array}$

&
$\Leftrightarrow$ & 
\ytableausetup{boxsize=1.4em,centertableaux}
$T$ $\quad$ = $\quad$ \ytableaushort{{*(gray!50)1}{*(gray!50)1}{*(gray!50)2}5,{*(gray!50)3}44,555}
\end{tabular}
\caption{The correspondence between the $k$-th line of a Gelfand-Tsetlin pattern and sub-tableau of an SSYT consisting of the entries $\leq k$. Here   $\la = (4,3,3,0,0)$ and $y^3 = (3,1,0)$.}\label{f:GT-SSYT}
\end{figure}
The total number of tilings in $T^f_{n,m}$ is equal to the total number of SSYTs with entries $1,\ldots,n$ and shapes fitting inside $m^n$. The number of SSYTs with entries $1,\ldots,n$ of a given shape $\lambda$ is just $s_{\la}(\underbrace{1,\ldots,1}_n)$ by the combinatorial definition of Schur functions and so the total number of tilings in $T^f_{n,m}$ is exactly $\phi_m(\underbrace{1,\ldots,1}_n)$. The number of tilings whose horizontal lozenges on the $k$-th diagonal are at positions $y^k =(y^k_1,\ldots,y^k_k)$, via the bijection with SSYTs described above, is equal to 
$$\sum_{\lambda \subset (m^n)} s_{y^k}(\underbrace{1,\ldots,1}_k) s_{\la/y^k}(\underbrace{1,\ldots,1}_{n-k})$$
and thus the probability of the positions of the horizontal lozenges on the $k$-th vertical line being $y^k$ is just that last quantity divided by $\phi_m(1^n)$. 

Observe that evaluation of the multivariate Bessel function is nothing else but the ratio of two Schur functions, since the determinant matches the determinant in Weyl's determinantal formula, and the ratio of the Vandermondes is Weyl's dimensional formula, i.e.

\begin{align*}
B_k&(x; \frac{y^k-\frac{m}{2} +\delta_k}{\sqrt{n}} ) \\
&= \frac{ \prod_i \exp(-\frac{x_i m}{2\sqrt{n} }) \det\left[ (\exp(\frac{x_i}{\sqrt{n}}))^{y^k_j +k-j} \right]_{i,j=1}^k }{ \prod_{i<j}(\frac{x_i}{\sqrt{n}} - \frac{x_j}{\sqrt{n}}) } \frac{ \prod_{i<j}(j-i)}{\prod_{i<j} (( y^k_i+k-i) - (y^k_j +k-j) )}\\
&=\prod_i \exp(-\frac{x_i m}{2\sqrt{n} }) \frac{ s_{y^k}(e^\frac{x_1}{\sqrt{n}},\ldots,e^\frac{x_k}{\sqrt{n}}) }{ s_{y^k}(1^k)} \frac{\prod_{i<j} \left( e^{x_i/\sqrt{n}} - e^{x_j/\sqrt{n}}\right)}{\prod_{i<j}(\frac{x_i}{\sqrt{n}} - \frac{x_j}{\sqrt{n}})}. 
\end{align*}
Here $\delta_k = (k-1,\ldots,1,0)$, the staircase partition, and in the vector sums the constants are added to each term. 
Notice that the last factor is $1+ o(1/\sqrt{n})$ for any fixed $x_1,\ldots,x_k$. 

We will use the following well-known identity, which is easy to prove from the combinatorial description of Schur functions ( see e.g. \cite[\S I.5]{Mac}):
 for any set of variables $z$ and $w$ and any partition $\alpha$ we have
$$s_{\alpha}(z,w) = \sum_{\beta} s_{\beta}(z)s_{\alpha/\beta}(w).$$

Now we  can calculate the moment generating function:
\begin{align*}
\mathbb{E}B_k\bl x;\frac{ Y^k_{n,m}-m/2}{\sqrt{n} }\br& \frac{\prod_{i<j}(\frac{x_i}{\sqrt{n}} - \frac{x_j}{\sqrt{n}})}{\prod_{i<j} \left( e^{x_i/\sqrt{n}} - e^{x_j/\sqrt{n}}\right)} \\
&=\sum_{y^k} \prod_i \exp(-\frac{x_i m}{2\sqrt{n} }) \frac{ s_{y^k}(e^\frac{x_1}{\sqrt{n}},\ldots,e^\frac{x_k}{\sqrt{n}})}{ s_{y^k}(1^k) }\frac{s_{y^k}(1^k)  \sum_{\lambda \subset (m^n)} s_{\la/y^k}(1^{n-k}) }{\phi_m(1^n)} \\
&=\prod_i \exp(-\frac{x_i m}{2\sqrt{n} })  \frac{1}{\phi_m(1^n)} \sum_{\lambda \subset (m^n) } \sum_{y^k} s_{y^k}(e^\frac{x_1}{\sqrt{n}},\ldots,e^\frac{x_k}{\sqrt{n}})s_{\la/y^k}(\underbrace{1,\ldots,1}_{n-k}) \\
&=\prod_i \exp(-\frac{x_i m}{2\sqrt{n} }) \frac{1}{\phi_m(1^n)} \sum_{\lambda \subset (m^n) } s_{\la}( e^\frac{x_1}{\sqrt{n}},\ldots,e^\frac{x_k}{\sqrt{n}}, \underbrace{1,\ldots,1}_{n-k} )\\
&
= \prod_i \exp(-\frac{x_i m}{2\sqrt{n} }) \Phi_m( e^\frac{x_1}{\sqrt{n}},\ldots,e^\frac{x_k}{\sqrt{n}};n ). \qedhere
\end{align*}

\end{proof}

\begin{lemma}\label{l:EBY}
$$\mathbb{E}B_k\bl x;\frac{Y^k_{n,m}-m/2}{\sqrt{n}}\br \to \exp\left( 1/16(a^2+2a)( x_1^2+\cdots x_k^2)\right)$$
as $n,m \to \infty $ with $m/n\to a$. 
\end{lemma}

\begin{proof}
This is a direct consequence of Proposition~\ref{lemma:EB=Phi} and the asymptotics of $\Phi_m$ from Theorem \ref{t:Phi_m_convergence}. 
\end{proof}

\begin{theorem}\label{t:gue}
Let $n,m\to \infty$ with  $m/n \to a$ for $a >0$, such that $\sqrt{n}(m/n -a) \to 0$. Then
$$ \frac{Y^k_{n,m}-m/2}{\sqrt{n(a^2+2a)/8}} \to \GUE_k$$
in the sense of weak convergence of random variables. Moreover, the so-rescaled positions of the horizontal lozenges on the first $k$ vertical lines, i.e. $\left\{  \frac{Y^j_{n,m}-m/2}{\sqrt{n(a^2+2a)/8}} \right\}_{j=1}^k$ weakly converge as random variables to the collection of eigenvalues $\{ \varepsilon^j\}_{j=1}^k$ of the principle submatrices from a $k\times k$ matrix from the $\GUE$ ensemble, where $\{\varepsilon^j\}$ are the eigenvalues of the submatrix formed by the first $j$ rows and columns. 
\end{theorem}
\begin{proof}
It is a classical result, following L\'evy's continuity theorem, that if the moment generating functions (MGF) $\mathbb{E}[e^{X_it}], i=1,\ldots$ of  a sequence of random variables $\{X_i\}_{i=1,\ldots}$ converge uniformly in a compact interval of $t$ (not necessarily containing 0) to the MGF of a given random variable $X$, then $X_i\to X$ in distribution (i.e. weakly); see e.g. \cite[Exercise 30.4]{Bi}. This statement easily generalizes when replacing the random variables by vectors of random variables and MGF by $\mathbb{E}B_k$.  The Theorem now follows by applying this fact to the sequences $Y^k_{n,m}$, namely we have that 
\begin{multline*}
\mathbb{E}B_k\bl x;\frac{Y^k_{n,m}-m/2}{\sqrt{n(a^2+2a)/8}}\br =  \mathbb{E}B_k\bl\frac{x}{\sqrt{(a^2+2a)/8}};\frac{Y^k_{n,m}-m/2}{\sqrt{n}}\br \\
\to \exp\left( \frac12( x_1^2+\cdots x_k^2)\right)=\mathbb{E}B_k(x;\GUE_k),\end{multline*}
since $B_k(x;y\alpha) = B_k(\alpha x; y)$ for any constant $\alpha$ and Lemma \ref{l:EBY} gives the asymptotics. 

The convergence of the collection of horizontal positions $\left\{  \frac{Y^j_{n,m}-m/2}{\sqrt{n(a^2+2a)/8}} \right\}_{j=1}^k$ to the collection of eigenvalues of the submatrices follows from the Gibbs property satisfied by both collections. Namely, given the vector $Y^k_{n,m}$, the distribution of the horizontal lozenges on the first $k-1$ vertical lines is clearly uniform subject to the interlacing conditions. The same is true (see~\cite{Bar}) for the eigenvalues $\{\varepsilon^j\}_{j=1}^{k-1}$ given  $\varepsilon^k$, which are again subject to the interlacing conditions, i.e. $\varepsilon^j_{i-1}\geq \varepsilon^{j-1}_{i-1} \geq \varepsilon^j_i$. 
\end{proof}

We can now compare  the results for the free boundary case and the hexagon.

\begin{proposition}\label{p:hex_gue}
Let $\alpha(n) = (m^n, 0^n)$ and $y\in \mathbb{R}$. Suppose that $\lim m/n=a$ is finite nonzero as $n\to \infty$. Then 
$$S_{\alpha(n)}(e^{y/\sqrt{2n}};2n) = \exp\bl \sqrt{2n}a/4 y +1/2( a^2/16+a/8)y^2 +o(1) \br.$$
\end{proposition}
 \begin{proof}
 Follows directly from Proposition~\ref{Prop_convergence_GUE_case}, applied with $N=2n$ and $f(t)=a/2$ for $t\in[0,1/2)$ and $f(t)=0$ for $t\in[1/2,1]$.
\end{proof}
We apply  Theorem~5.1 of~\cite{GP} with the specific asymptotics for $S_{\alpha(n)}$ just obtained. Here $\alpha(n)$ corresponds to the positions of the horizontal lozenges on line $2n$ from the left, and it gives the hexagonal domain  with sides $m\times n \times n \times m \times n \times n$. We get the following specific result
\begin{theorem}[\cite{GP, JN, N, Nov}]
Let $\Upsilon^k_{n,m}$ denote the positions of the horizontal lozenges on the $k$-th vertical line (from the left) from a uniformly random tiling of the $m\times n \times n \times m \times n \times n$ hexagon ($m$ is the length of the vertical side). Let $m/n \to a$ as $n \to \infty$. Then
$$ \frac{ \Upsilon^k_{n,m} - n a/2}{\sqrt{n (a^2+2a)/8 } } \to \GUE_k.$$
\end{theorem}
This result is identical to the one given in Theorem~\ref{t:gue}, so we have:
\begin{corollary}\label{c:gue}
The joint distribution for the (shifted and rescaled) positions of the horizontal lozenges on line $k$ near the left vertical boundary for free boundary tilings of the half-hexagon ($Y^k_{n,m}$), and the fixed boundary tilings of the full hexagon ($\Upsilon^k_{n,m}$) are the same as $n,m\to \infty$ with $m/n\to a$.
\end{corollary}

\begin{proof}[Alternative proof of Proposition \ref{l:gue}]

For the sake of self-containment we provide an elementary proof of that proposition. 
We use the well-known formula for the distribution of the eigenvalues \eqref{P_GUE}. 
We then calculate 

 \begin{align} \label{multi_gue_int}
\mathbb{E}B_k(x;\GUE_k) =\int_{-\infty}^{+\infty}\cdots \int_{-\infty}^{+\infty} B_k(x;y) P_{\GUE_k}(y_1,\ldots,y_k)dy_1\ldots dy_k.
\end{align}
We now consider the integrand above, and observe that 
\begin{align}\label{gue_det}
B_k(x;y) P_{\GUE_k}(y_1,\ldots,y_k) = \frac{1}{Z_{2,k}} \frac{ \det \left[ \exp(x_iy_i-y_i^2) \right]_{i,j=1}^k \prod_{i<j}(y_i-y_j) }{\prod_{i<j} (x_i-x_j) }\prod_{i<j}(j-i)
\end{align}
When we expand the product of the determinant and the Vandermonde, we get summands of the form 
$$\prod_j \exp(x_{i_j}y_j - y_j^2) y_j^{m_j},$$
where $\iota=(i_1,\ldots,i_k)$ is a permutation of $(1,\ldots,k)$ and $\sigma=(m_1,\ldots,m_k)$ is a permutation of $(0,1,\ldots,k-1)$. 
For any $x$ and integer $m$, with $u=(y-x/2)$, we have that 
$$\int_{-\infty}^{+\infty} \exp(xy-y^2) y^m dy = \exp(x^2/4)\sqrt{\pi}\int_{-\infty}^{+\infty} \frac{ \exp(-u^2 )}{\sqrt{\pi}} (u+x/2)^m du.$$
Expand $(u+x/2)^m$ as a polynomial in $u$ and note that for any nonnegative integer $\ell$ the integral $\int \frac{e^{-u^2}}{\sqrt{\pi}} u^\ell du$  is always finite 
and is equal to 1 when $\ell=0$. So we can write 
\begin{equation}\label{int_gue} \int_{-\infty}^{+\infty} \exp(xy-y^2) y^m dy = \exp(-x^2/4)\sqrt{\pi} ( (x/2)^m + q_m(x) ),\end{equation}
where $q_m(x)$ is a certain polynomial in $x$ of degree at most $m-1$.

We now expand \eqref{multi_gue_int}, by substituting the formula \eqref{gue_det}, expanding the determinant and $\Delta(y)$ as signed sums over permutations, and then substituting the formula \eqref{int_gue} (here $\iota=(i_1,\ldots,i_k)$ and $\sigma=(m_1,\ldots,m_k)$ go through all permutations described above):
\begin{multline*}
Z_{2,k} \mathbb{E}B_k(x;\GUE_k) =  \frac{\prod_{i<j}(j-i) }{\prod_{i<j} (x_i-x_j)} \prod_{i}\left(  \exp\left( \frac{x_{i}^2}{4}\right) \sqrt{\pi } \right)\times\\
 \times\left[ \sum_{\iota, \sigma} \sgn(\iota)\sgn(\sigma)\prod_j \left( (\frac{x_{i_j}}{2})^{m_j} +q_{m_j}(x_{i_j}) \right) \right]\\
 =k! \frac{\prod_{i<j}(j-i) }{\prod_{i<j} (x_i-x_j)} (\sqrt{\pi })^k\prod_{i} \exp\left( \frac{x_{i}^2}{4}\right)
 \left[ \sum_{\sigma'} \sgn(\sigma')\prod_j \left( (\frac{x_{j}}{2})^{m_j} +q_{m_j}(x_{j}) \right) \right] \\
 = k! \frac{\prod_{i<j}(j-i) }{\prod_{i<j} (x_i-x_j)} (\sqrt{\pi })^k\prod_{i} \exp\left( \frac{x_{i}^2}{4}\right) \det\left[ (\frac{x_{j}}{2})^{i-1} +q_{i-1}(x_{j}) \right]_{i,j=1}^k.
\end{multline*}
Here the double summation over the two permutations $\iota, \sigma$  was replaced by a summation over their product ($\iota\sigma = \sigma'$) with a factor of $k!$. The last determinant is easily seen to be $2^{-\binom{k}{2}} \Delta(x)$ as it is a polynomial of maximal degree $\binom{k}{2}$ with leading coefficient $2^{-\binom{k}{2}}$ and divisible by all monomials $x_i-x_j$. Thus we get
$$\mathbb{E}B_k(x;\GUE_k) = \frac{1}{Z_{2,k}}k!\prod_{1\leq i<j \leq k}(j-i) \left( \pi \right)^{k/2} 2^{-\binom{k}{2}}\exp\left( \frac{1}{4}(x_1^2+\cdots+x_k^2) \right)$$
Substituting the formula for $Z_{2,k}$, the constants cancel and we get the desired formula. 
\end{proof}

\section{Limit shape}\label{s:limit}

We now show the existence of a limit shape for the height function (alternatively, the symmetric plane partition) of the tilings with free boundary, which does not follow immediately from other previously considered cases. 
Kenyon, Okounkov, Sheffield \cite{KOS}, following \cite{Sh}, show that uniformly random dimer covers on a bipartite planar locally periodic graph exhibit limit shapes of their height functions. In the current case, lozenge tilings in a fixed domain correspond to dimer covers on a hexagonal grid.  However, the present free boundary tilings are not directly dimer covers on a \emph{locally periodic planar} graph with uniform probability measure, and even though they can be considered as symmetric dimer covers of the mirrored graph (see Figure~\ref{f:sym_pp}) the framework of \cite{KOS} does not  apply directly. Alternatively, following \cite{CKP} and \cite{KOS}, the  variational principle could be made rigorous and prove formally the existence of a limit shape, but here we choose  the methods coming from the asymptotic analysis of symmetric polynomials. 

We show that the positions of the horizontal lozenges at any vertical line $x=\alpha n$, given by a signature (partition) $\mu^{\alpha n}$, converge to a ``limit shape'' in the following sense. 
We show that there is a monotone piecewise continuous function $f_{\alpha}(t):[ 0, b ] \to \mathbb{R}$ for some sufficiently large $b$, such that
$$\lim_{n\to \infty} \frac{ \mu^{\alpha n}_i +\alpha n-i}{\alpha n} \to f_\alpha \left( \frac{i}{\alpha n}\right)$$
in probability (see e.g. \cite[Section 6.2]{KOS}), where $\mu^{\alpha n}$ is now a random signature of length $\alpha n$, according to some probability measure on signatures. To formalize this and avoid the issues with discontinuities, for any signature $\lambda$ we define the function from \cite{BBO}
$$w_{\lambda}(x) = 
\begin{cases} 
2(i-1)+x, &\text{ if $x\in[ \lambda_i-i+1,\lambda_{i-1}-i+1]$ }\\
2\lambda_i-x, &\text{ if $x\in[\lambda_i-i,\lambda_i-i+1]$}.
\end{cases}$$
Visually, the graph of this function traces the border of the reflected about the diagonal and then rotated $135^\circ$ clockwise Young diagram of $\lambda$, including the bisectors of quadrants I and II as $x$ leaves the shadow of the diagram.

  Moreover, the ``limit shape'' of the random  tilings as $n \to \infty$ should be a function $H:\mathbb{R}^2 \to \mathbb{R}$. At a point $(\alpha,t)$, this is $H(\alpha,t) =f_{\alpha}(t)$, which can also be considered upon the reflection-rotation mentioned above, when it would be ``the limit'' of $\frac{1}{N}w_{\lambda}(Nx)$.

 For any signature $\lambda$ of length $N$, define the counting measure 
$$m[\lambda] = \frac{1}{N} \sum_{i=1}^N \delta\left( \frac{\lambda_i+N-i}{N}\right),$$
where $\delta$ is the Dirac delta measure. Clearly, $m[\lambda]$ is a compactly supported probability measure. 

The following statement, \cite[Proposition 2.2]{BBO},  gives a formal basis to the current approach.
\begin{proposition}[\cite{BBO}]\label{p:limit_measure}
Suppose that for every $N=1,\ldots$,  we have an ensemble $\mu^N$ of random signatures distributed according to some probability measure on length-$N$ signatures. Suppose that the corresponding random measures $m[\mu^N]$ converge weakly, in probability, to a nonrandom probability measure $m$ with support in a bounded interval $[b_1,b_2] \subset \mathbb{R}$. 

(i) Then $m$ is absolutely continuous w.r.t the Lebesgue measure, so it has a density $p(x)$ vanishing outside the interval. 

(ii) The random functions $\frac{1}{N}w_{\mu^N}(Nx)$ uniformly converge in probability to a nonrandom function $w(x)$, determined uniquely by:
$w(x)=x$, for $x>b_2$, $w(x)=x+2$ for $x<b_1$, and $w'(x) = 1- 2p(x)$ almost everywhere on $[b_1,b_2]$.
\end{proposition}
In other words, the limit shape is actually the distribution function of the measure $m$, if such exists. 
 
 We will now 
use the following setup and result from~\cite{BG} to prove the convergence to a nonrandom measure as in Proposition~\ref{p:limit_measure}.

For any real $\alpha \in (0,1)$ and signature $\lambda$ of $n$ parts, define for any signature $\mu=(\mu_1\geq \cdots \geq \mu_{\lfloor \alpha n \rfloor})$ of $\lfloor \alpha n\rfloor$ parts, the following probability
\begin{equation}\label{eq:Palpha}
P^{\alpha}_{m,n}(\mu) = \frac{\sum_{\lambda: \lambda_1\leq m} s_{\lambda/\mu}(1^{n-\lfloor \alpha n\rfloor}) s_{\mu}(1^{\lfloor \alpha n\rfloor})}{\sum_{\lambda: \lambda_1\leq m} s_{\lambda}(1^n)}.
\end{equation}
Combinatorially this is the probability that a random SSYT with at most $m$ columns and $n$ letters will have the first $\lfloor \alpha n\rfloor$ letters forming an SSYT of shape $\mu$. Thus, it is evident that summing over all possible shapes $\mu$, we obtain all SSYT of shape $\lambda$ and $n$ letters, and thus $P^{\alpha,\lambda}$ is a probability measure on the signatures of length $\lfloor \alpha n\rfloor$.

Finally, for any probability measure $\rho$ on the set of signatures of length $N$, define the normalized $GL_N$-character generating function
$$S_{\rho}(u_1,\ldots,u_N) = \sum_{\mu: \ell(\mu)=N} \rho(\mu) \frac{s_{\mu}(u_1,\ldots,u_N)}{s_{\mu}(1^N)}.$$
Given a probability distribution $\rho$ on signatures $\lambda$, we denote (by a slight abuse of notation, to agree with \cite{BG}) by $m[\rho]$ the random counting measures $m[\lambda]$, where $\lambda\sim\rho$. 

We will apply the following result, proven in~\cite{BG} and inspired by~\cite{BBO}.

\begin{theorem}[Theorem 5.1 in \cite{BG}]\label{t:BG}
For each $N$, let $\rho^N$ be a measure on the set of signatures of length $N$. Suppose that for every $k$
$$\lim_{N \to \infty} \frac{1}{N}\ln\left( S_{\rho^N}(u_1,\ldots,u_k,1^{N-k}) \right) = Q(u_1)+\cdots+ Q(u_k),$$
where $Q$ is an analytic function in a neighborhood of 1 and the convergence is uniform in an open (complex) neighborhood of $(1,\ldots,1)$. Then the random pushforward measures $m[\rho^N]$ converge, as $N \to \infty$, in probability, in the sense of moments, to a {\rm deterministic} measure $M$ on $\mathbb{R}$, whose moments are given by
$$\int_{\mathbb{R}} x^r M(dx) = \sum_{\ell=0}^{r} \binom{r}{\ell}\frac{1}{(\ell+1)!} \frac{\partial^\ell}{\partial u^\ell} u^r Q'(u)^{r-\ell} \Bigg|_{u=1}\; .$$

\end{theorem}

We apply this theorem to the measure $\rho=P^{\alpha}_{m,n}$. This measure is the distribution of the partition $\mu$, which represents the positions (height function) of the horizontal lozenges at the vertical line $k=\lfloor \alpha n \rfloor $, i.e. $y^k$ as in Section~\ref{s:GUE} (see Figure~\ref{f:y_k}), of a uniformly random lozenge tiling from $T^f_{n,m}$. 
We have that $N=\lfloor \alpha n\rfloor$, and then compute as in Section~\ref{s:GUE} 

\begin{equation}\label{S_Palpha}
\begin{split}
S_{P^{\alpha}_{m,n}}&(u_1,\ldots,u_N) = \sum_{\mu: \ell(\mu)=N} P^{\alpha}_{m,n}(\mu) \frac{s_{\mu}(u_1,\ldots,u_N)}{s_{\mu}(1^N)}\\
&= \sum_{\mu: \ell(\mu)=N} \frac{\sum_{\lambda: \lambda_1\leq m} s_{\lambda/\mu}(1^{n-N}) s_{\mu}(1^{N})}{\sum_{\lambda: \lambda_1\leq m} s_{\lambda}(1^n)}  \frac{s_{\mu}(u_1,\ldots,u_N)}{s_{\mu}(1^N)}\\
&=\frac{ \sum_{\lambda: \lambda_1\leq m} \sum_{\mu} s_{\lambda/\mu}(1^{n-N}) s_{\mu}(u_1,\ldots,u_N)  }{\sum_{\lambda: \lambda_1\leq m} s_{\lambda}(1^n) } =\frac{ \sum_{\lambda: \lambda_1\leq m} s_{\lambda}(u_1,\ldots,u_N,1^{n-N} )  }{\sum_{\lambda: \lambda_1\leq m} s_{\lambda}(1^n) }\\
&= \Phi_m(u_1,\ldots,u_N;n).
\end{split}
\end{equation}

To state the next proposition we need to define the following functions of $y$, which are analytic in a neighborhood of $y=0$. Although they depend only on $x=e^y$, they will be used as functions of the argument $y$:
\begin{multline}\label{def:Hs}
h(y) =\frac{1}{4}\left( (x+1)+\sqrt{(x+1)^2+ (a^2+2a)\left(x-1\right)^2} \right)\\
\phi(y;a)=\bl\frac{a}{4}+1\br\ln\left(h(y)-\bl\frac{a}{4}+1\br(x-1)\right ) -\bl\frac{a}{4}+\frac{1}{2}\br\ln\left(h(y)-\bl\frac{a}{4}+\frac{1}{2}\br(x-1)\right ) \\ + \frac{a}{4}\ln\left(h(y)+\frac{a}{4}(x-1)\right) -\bl\frac{a}{4}-\frac{1}{2}\br\ln\left(h(y)+\bl\frac{a}{4}-\frac{1}{2}\br(x-1)\right).
\end{multline}
\begin{proposition}\label{p:ln_phi}
Let $m,n\to \infty$ with $m/n \to a$, where $a$ is a positive real number. Then for any fixed $k$, we have
$$\lim_{n \to \infty} \frac{1}{n} \ln \Phi_m(u_1,\ldots,u_k;n) =\Psi_a(u_1)+\cdots+\Psi_a(u_k),$$
where $\Psi_a(e^y)=y\frac{a}{2}+2\phi(y;a).$
\end{proposition}
\begin{proof}
We first show the statement for $k=1$. 
We have that 
\begin{equation}\label{eq:phi_m_S}
\Phi_m(u;n) = u^{m/2} \frac{\X_{\tau^m}(u;n)}{\X_{\tau^0}(u;n)}= u^{m/2} \frac{S_{\nu^m}(u;2n)}{S_{\nu^0}(u;2n)}.
\end{equation}
By Proposition~\ref{proposition_convergence_mildest} ( and its validity over complex domains,  \cite[Prop. 4.7]{GP}) we have that for sequences of partitions $\lambda(N)$ satisfying certain regularity constraints,
\begin{equation}\label{eq:lim_S}\lim_{N \to \infty} \frac{\ln S_{\lambda(N)}(e^y;N)}{N} = yw_0 - \mathcal{F}(w_0) -1 - \ln(e^y-1),\end{equation}
where the convergence is uniform on compact complex domains for $y$.
In our particular case of $\lambda(N)=\nu^m$ we have the limit profiles $f_m(t) = \frac{a}{4}$ for $t\in[0,\frac{1}{2}]$ and $f_m(t) = -\frac{a}{4}$ for $t \in (\frac{1}{2},1]$, and for $\lambda(N)=\nu^0$ we have $f_0=0$ for all $t$. 
Then 
\begin{multline*}
\mathcal{F}(w;f_m) = \int_0^{\frac{1}{2}} \ln(w - \frac{a}{4}-1+t)dt + \int_{\frac{1}{2}}^1 \ln(w+\frac{a}{4}-1+t)dt =\\
(w-\frac{a}{4}-\frac{1}{2})\ln(w-\frac{a}{4}-\frac{1}{2})-(w-\frac{a}{4}-1)\ln(w-\frac{a}{4}-1)  + (w+\frac{a}{4})\ln(w+\frac{a}{4})  - (w+\frac{a}{4}-\frac{1}{2})\ln(w+\frac{a}{4}-\frac{1}{2})-1
\end{multline*}
and $w_0$ is a root of the equation
$$y= \frac{\partial}{\partial w} \mathcal{F}(w;f_m) = \ln(w-\frac{a}{4}-\frac{1}{2})-\ln(w-\frac{a}{4}-1)+\ln(w+\frac{a}{4})-\ln(w+\frac{a}{4}-\frac{1}{2}). $$ 
In particular, when $y\not \in (-\infty, 0]$, using definition~\eqref{def:Hs}, we can choose
$$w_0 = \frac{1}{2} + \frac{h(y)}{(e^y-1)}.$$
We have that $h(y)$ is analytic and $h(0)=1$. In order to claim that the asymptotics holds for complex values of $y$, we invoke~\cite[Prop. 4.7]{GP}, which states that Proposition~\ref{proposition_convergence_mildest} holds if a certain steepest descent contour exists in the original integral formula for $S_{\lambda(N)}$. In our particular case we have that $w_0\not \in (-\infty,a/4+1]$, everything is well defined and this steepest descent contour passes on the right side of the poles as is necessary. 

By collecting the terms containing $e^y-1$, we rewrite equation~\eqref{eq:lim_S} as
\begin{multline*}
\lim_{N \to \infty} \frac{\ln S_{\nu^m}(e^y;N)}{N} 
= yw_0 - w_0\ln\left( \frac{ ( w_0-\frac{a}{4}-\frac{1}{2} )(w_0 +\frac{a}{4})}{ (w_0 -\frac{a}{4}-1)(w_0+\frac{a}{4}-\frac{1}{2})} \right) 
\\
+\ln(e^y-1)\left( -\frac{a}{4} -\frac{1}{2} +\frac{a}{4}+1 +\frac{a}{4} -\frac{a}{4}+\frac{1}{2}\right)-\ln(e^y-1) - \phi(y;a)
=\phi(y;a),
\end{multline*}
where  $\phi(y;a)$ is defined in equation~\eqref{def:Hs}.
Substituting it into equation~\eqref{eq:lim_S} we obtain the desired limit. Moreover, since both sides are analytic around $y=0$, the equality extends to a neighborhood of $0$.

In the case of $m=0$, i.e. for $f_0$ we have a standard calculation 
$$
\mathcal{F}(w;f_0) = \int_0^1\ln(w-1+t)dt = w\ln(w)-(w-1)\ln(w-1)-1
$$
with the corresponding root $w_0' = \frac{e^y}{e^y-1}$ and limit
$$\lim_{N\to \infty} \frac{\ln S_{\nu^0}(e^y;N)}{N} =yw_0' - w_0'y -\ln(w_0'-1)  -\ln(e^y-1)= 0,$$
which is also justified in \cite[Example 1]{GP}. 

By Proposition~\ref{Proposition_Schur_Simplectic_1}, the same limits hold when replacing $S_{\nu}$ by the corresponding $\X_{\tau}$.

Thus we have that 
$$\lim_{n\to \infty} \frac{\ln \Phi_m(e^y;n)}{n} =\lim_{n\to \infty}\left( y\frac{m}{2n} + \frac{\ln S_{\nu^m}(e^y;2n)}{n}- \frac{\ln S_{\nu^0}(e^y;2n)}{n}\right)= y\frac{a}{2} +2\phi(y;a).$$
Applying the just obtained limits for $\ln\X_{\tau}(e^y;n)/n$ in 
 Proposition~\ref{p:ln_mult}, we obtain the corresponding limit for the multivariate characters $\X_{\tau^m}$ and $\X_{\tau^0}$, and thus
\begin{multline*}
\frac{1}{n} \ln \Phi_m(e^{y_1},\ldots,e^{y_k};n)  = \sum_i y_i\frac{m}{2n} + \frac{1}{n}\ln\X_{\tau^m}(e^{y_1},\ldots,e^{y_k};n) -\frac{1}{n}\ln\X_{\tau^0}(e^{y_1},\ldots,e^{y_k};n) \\ \to \sum_i y_i\frac{a}{2} + \sum_i 2\phi(y_i;a), \qquad \text{ uniformly as $n\to \infty$,}
\end{multline*}
which completes the proof.
\end{proof}

\begin{theorem}\label{t:limit_shape}
Let $n,m \in \mathbb{Z}$, such that $m/n \to a$ as $n \to \infty$, where $a \in(0,+\infty)$. Let $H(x,y)$ be the height function of the uniformly random lozenge tiling in $T^f_{n,m}$. Then as $n\to \infty$, for all $x\in (0,1), \; y \geq 0$ the normalized $\frac{1}{n}H(nx,ny)$ converges in probability to a deterministic function $L(x,y)$, referred to as ``the limit shape''. 

Moreover, for any fixed $x \in (0,1)$, the function $L(x,y)$ is the distribution function of the limit measure $M$ whose moments are given by
$$ \int_{\mathbb{R}} t^r M(dt) = \sum_{\ell=0}^{r} \binom{r}{\ell}\frac{1}{(\ell+1)!} x^{-r+\ell}\frac{\partial^\ell}{\partial u^\ell} u^r \Psi'_a(u)^{r-\ell}\Bigg|_{u=1},$$
where $\Psi_a(u)$ is defined in Proposition~\ref{p:ln_phi}.
\end{theorem}

\begin{proof}
The distribution $\rho^N(\mu) = P^{\alpha}_{m,n}(\mu)$ defined in equation~\eqref{eq:Palpha} is a probability distribution on the set of signatures of length $N=\lfloor x n\rfloor$, as explained there. Further, the corresponding $GL_N$-character generating function $S_{\rho^N}$ is, by equation~\eqref{S_Palpha}, equal to $\Phi_m$. By Proposition~\ref{p:ln_phi},  this $S_{\rho^N}$ satisfies the conditions of Theorem~\ref{t:BG}, noting that $\frac{1}{N}=\frac{1}{x} \frac{1}{n}$, so the random measures $m[\mu]$, defined by the random signatures $\mu$ distributed according to $P^{\alpha}_{m,n}$, converge in probability to a deterministic measure $M$, defined accordingly through its moments. Finally, Theorem~\ref{t:BG} applies to this measure $M$ with $Q(u) = \frac{1}{x} \Psi_a(u)$ and we obtain the desired limit of the height function $L(x,y)$. For any fixed $x$, $L(x,y)$ as a function of $y$ is equal to  the distribution function for the measure $M$.
\end{proof}
\begin{corollary}\label{c:limit}
The limit shape for tilings of the free boundary  half-hexagon domain is the same as the corresponding half of the limit shape for tilings of the full $m\times n \times n \times m \times n \times n$ hexagon, first given in~\cite{CLP}.
\end{corollary}
\begin{proof}
We use the result of~\cite{BG} which proves the limit shape for fixed boundary tilings and derives the moments of the corresponding measure $M$ as in Theorem~\ref{t:limit_shape} using Theorem~\ref{t:BG}. For tilings of the full hexagon we take $\lambda(2n) =(m^n, 0^n)$ as the boundary on the right side of the domain, i.e. $\lambda(2n)+\delta_{2n}$ are the positions of the fixed horizontal lozenges on the $2n$-th vertical line from the left. 
In this case we have $N=\lfloor x n \rfloor$ for vertical section at distance $x$ from the left and
$$S_{\rho^N}(u_1,\ldots,u_k,1^{N-k}) = S_{\lambda(2n)}(u_1,\ldots,u_k;2n),$$
When $k=1$ we have that 
$$Q_s(u) =\frac{1}{x} \lim_{n \to \infty} \frac{S_{\lambda(2n)}(u;2n)}{n}.$$
By the approximate multiplicativity of the normalized Schur functions (as in Proposition~\ref{p:ln_mult}), shown in~\cite{GP}, we also have  that
$$\lim_{N \to \infty} \frac{1}{N} \ln\left(S_{\rho^N}(u_1,\ldots,u_k,1^{N-k}) \right) = \lim_{N \to \infty} \frac{S_{\lambda(2n)}(u_1,\ldots,u_k;N)}{N} = \sum_{i=1}^k \lim_{N\to \infty} \frac{S_{\lambda(2n)}(u_i)}{N}$$
and the condition in Theorem~\ref{t:BG} holds. Hence, the limit shape on the line $\alpha$ is determined by the function $Q_s$ above. 
Going back to the proof of Proposition~\ref{p:ln_phi}, we just need to compare the corresponding limits for $S_{\nu^m}(u;2n)$ and $S_{\lambda(2n)}(u;2n)$ in $Q=\alpha^{-1} \Psi_a(u)$ and $Q_s(u)$. But $\lambda(2n) = \nu^m + ( (m/2-1/2)^n )$, so by the integral formula~\eqref{f:int} we see that
$$S_{\nu^m}(u;2n) = u^{-m/2+1/2} S_{\lambda(2n)}(u;2n).$$
Substituting this formula in equation~\eqref{eq:phi_m_S}, we have obtain that $\Phi_m(u;n)$ is equal to $S_{\lambda(2n)}$ times some bounded factors. Taking logarithms, dividing by $n$ and passing through the limit,  we see that $Q(u) =  Q_s(u)$. Hence the moments at the vertical line $x$  coincide in both cases and so does the limit shape in the common region. 
\end{proof}

\section{Asymptotics II: when $\lim_{n\to \infty} \frac{m}{n}=0,\infty$}\label{s:other}

So far, for reasons concerning the physics nature of the models, the interest has been in studying the ``scaling limits'' of lattice models, in which the scaling factors are the same in all directions. Here we consider other regimes, in which the scalings in the different directions differ in growth order. For example, the vertical scaling is proportional to $m\sim \sqrt{n} $ and the horizontal is $n$.

In the usual case the scaling limit is such that both horizontal and vertical sides are linearly proportional to $n$ (the vertical scaling here is $m$ and it is about $an$, for a constant $a$) as considered in the previous Sections, another extreme of a scaling regime is when $n$ is fixed, and $m \to \infty$. If $n=1$, then there is only one horizontal lozenge at position $y^1_1$ and it is uniformly distributed in $\{0,1,\ldots,m\}$. If $n>1$, but still fixed, the distribution can still be computed using the MGF $\Phi_m$ as in Section~\ref{s:GUE}, but the formulas become more complicated and will be skipped here.

Suppose now that $m, n\to \infty$, but $m\gg n$, i.e. $m/n \to \infty$. This corresponds to the case $a=\infty$ above. We investigate the distribution of the positions of the horizontal lozenges at the $k$ left-most vertical lines, where $k$ is a fixed number. It is clear from the symmetry of the model that the mean of the left-most horizontal lozenge is $\frac{m}{2}$. It is not, however, a priori obvious what the correct scaling of these positions should be, in order to get a limiting distribution for $Y^k_{n,m}$ as $m,n\to \infty$. The following proposition is complementary to Theorem~\ref{t:gue}.

\begin{proposition}\label{p:n_small}
Suppose that $m,n\to \infty$ with $m/n \to \infty$, and suppose that $\frac{n^2}{m}=o(1)$. Then the shifted rescaled positions ($Y^k_{n,m}$) of the horizontal lozenges on the $k$th vertical line satisfy:
$$ \frac{ Y^k_{n,m}-m/2 }{m/\sqrt{8n}} \to \GUE_k$$
in distribution for all fixed $k$. The collection $\{\frac{ Y^j_{n,m} -m/2}{m/\sqrt{8n} } \}_{j=1}^k$ converges to the GUE-corners process.
\end{proposition}

Note that when $m= an$ the variance coincides with the leading (in $a$) term, that is $a^2/8 n$, of the variance from Theorem~\ref{t:gue}. Moreover, similar statement holds as long as $n=o(m)$, but since the asymptotic analysis in the steepest descent approach below becomes more involved with extra no longer negligible factors in the asymptotics towards the function $\F$ and the computations afterwards, we limit ourselves to $n=o(\sqrt{m})$. 
\begin{proof}
Let first $k=1$. Because of the relationship between $\Phi_m$ and the normalized symplectic characters $\X$, and then between $\X$ and the normalized Schur functions $S$, we first derive the asymptotics for $S$ in the given cases. 

We use the explicit integral formula~\eqref{f:int}, where we set $u=z+\frac{1}{2}$ and obtain
\begin{multline}
S_{\nu^m}(e^{y\sqrt{n}/m};2n) \\
= \frac{(2n-1)!}{(e^{y\sqrt{n}/m}-1)^{2n-1} }\frac{1}{2\pi \ii} \int \frac{e^{zy\sqrt{n}/m}}{\prod_{i=1}^n (z - (\frac{m}{2}+\frac12+2n-i)) (z -(-\frac{m}{2}+\frac12 + n-i)) }dz\\
= \frac{(2n-1)!e^{-\frac{1}{2}y\sqrt{n}/m}}{(e^{y\sqrt{n}/m}-1)^{2n-1} }\frac{1}{2\pi \ii} \int \frac{e^{uy\sqrt{n}/m}}{ \frac{ \Gamma(u- (\frac{m}{2}+n) )\Gamma(u-(-\frac{m}{2} ))}{ \Gamma(u - (\frac{m}{2}+2n) )\Gamma(u-(-\frac{m}{2}+n))}}du\\
= \frac{(2n-1)!e^{-\frac{1}{2}y\sqrt{n}/m}}{(e^{y\sqrt{n}/m}-1)^{2n-1} } \frac{1}{2\pi \ii} \int \frac{\exp\left( uy\sqrt{n}/m +
 2n - f(u,m,n) \right) }{ \sqrt{ \frac{ (u-(\frac{m}{2}+2n))(u-(-\frac{m}{2}+n))}{(u - (\frac{m}{2}+n))(u-(-\frac{m}{2} ))} } }\left(1 + O(1/m)\right)du,
\end{multline}
where we used Stirling's approximation and set
\begin{multline*}
f(u,m,n)= (u - \frac{m}{2}-n) \ln(u - \frac{m}{2}-n) +(u+\frac{m}{2} )\ln(u+\frac{m}{2} ) \\
- (u - \frac{m}{2}-2n)\ln(u - \frac{m}{2}-2n) -(u+\frac{m}{2} -n)\ln(u+\frac{m}{2} -n). 
\end{multline*}
The contour of integration should contain all poles $-m/2+1/2,\ldots,-m/2+n-1/2, m/2+n+1/2,\ldots,m/2+2n-1/2$. We use the branch of $\ln(z)$ defined outside the negative imaginary axis, i.e. outside $\{ \rm{Re}(z)=0, \rm{Im}(z)\leq 0\}$, and we see that the contour can easily be deformed to pass above $0$ and below the poles on both sides.  
Set $u=mw$ with $w$, such that $|w|>1$. Observe that for any $a$ and $b$, such that $|a|^2<<|b|$, we have
$$(a+b)\ln(a+b) 
= (a+b)\ln(b) +a+O(a/b).$$
Applying this to the summands in $f(mw,m,n)$ with $b = wm\pm m/2$ and $a$ equal to the corresponding linear function of $n$ we get
\begin{multline*}
f(mw,m,n) = (mw -\frac{m}{2} -n)\ln(mw-\frac{m}{2}) -n +O(n/m) +\\
+(mw+\frac{m}{2} )\ln(mw+\frac{m}{2}) 
- (wm - \frac{m}{2}-2n)\ln(wm - \frac{m}{2}) - (-2n) \\ -  (wm+\frac{m}{2} -n)\ln(wm+\frac{m}{2}) - ( -n)\\
=2n + 2n \ln(m) + n \ln(w -\frac12)  +n  \ln(w+\frac12) + O(n/m) 
\end{multline*}
Putting everything together, and observing that the square root in the integral above is of the order of $1 +O(n/m)$, we can rewrite the integral as
\begin{multline}\label{nm:schur1}
S_{\nu^m}(e^{y\sqrt{n}/m};n)  = \frac{(2n-1)!}{(e^{y\sqrt{n}/m}-1)^{2n-1} m^{2n-1}} \sqrt{n} \times \\
\times \frac{1}{2\pi \ii} \int \exp n\left( vy - ( \ln(v -\frac{1}{2\sqrt{n}})  + \ln(v+\frac{1}{2\sqrt{n}}) ) -\ln(n) \right) (1 + O(n/m) ) dv, 
\end{multline}
after setting  $w=\sqrt{n}v$. We can now apply the method of steepest descent to the above integral with
$$\F(v,y) = vy-\ln\left(v^2 - \frac{1}{4n}\right)$$
The critical point is given by the equation
$$y - \frac{2v}{v^2-1/4n} =0.$$
Set
$$
 v_0 = \frac{1}{y}\left(1+\sqrt{1+\frac{y^2}{4n} }\right).
$$

%

If $y$ is real positive, then the steepest descent contour is given by $v=v_0+\ii s$ for $s\in \mathbb{R}$ and all the poles of the original function are to the left of this contour. Note that around $v=0$ and the negative imaginary axis, $\F(v,y) \to -\infty$ and the contribution  to the integral from these parts is negligible. So we can apply the steepest descent method with such contour, to get that
\begin{multline*}
\frac{1}{2\pi \ii}\oint \exp\left(n\F(v,y)\right) \approx \frac{1}{2\pi \ii} \exp(n\F(v_0,y) )\int_{-\infty}^{+\infty} \exp (n(-\frac{s^2}{2} \frac{\partial^2}{\partial v^2 }\F(v_0,y) ))d (\ii s)\\
= \exp(n\F(v_0,y) ) \frac{1}{\sqrt{2\pi (n \frac{\partial^2}{\partial v^2} \F(v_0,y) )} }.
\end{multline*}
In this particular case we have 


$$\F(v_0,y) = 2 + \frac{y^2}{8n} -2\ln2 +2\ln(y) - \frac{y^2}{16n}  +O(n^{-2}), $$
and
$$\frac{\partial^2}{\partial v^2} \F(v,y)\Bigg|_{v=v_0} =\frac{y^2}{8nv_0^2} + \frac{y}{v_0} = \frac{y^4}{8n}(1-\frac{y^2}{8n} +O(n^{-2})) +\frac{y^2}{2}(1-\frac{y^2}{16n} + O(n^{-2})) .$$

Putting everything together in formula \eqref{nm:schur1} we get
\begin{multline}
S_{\nu^m}(e^{y\sqrt{n}/m};2n)  = \frac{(2n-1)!n^{-n}}{(e^{y\sqrt{n}/m}-1)^{2n-1} m^{2n-1}}
\times \frac{1}{2\pi \ii} \int \exp\left( n\F(v,y) \right) (1 + O(n/m) ) dv \\
= \frac{ (2n-1)! * \exp \left( -n\ln(n) + 2n+y^2/16 - 2n\ln(2) +2n\ln(y) +O(n^{-1}) \right)(1 + O (n/m)) }
{(e^{y\sqrt{n}/m}-1)^{2n-1} m^{2n-1} \sqrt{2\pi \left( \frac{y^2}{2} + \frac{y^4}{8n} +O(n^{-2})\right)}}\\
=\frac{\sqrt{2\pi 2n} \exp\left[ 2n\ln(2n)  - 2n -n\ln(n) + 2n+y^2/16 - 2n\ln(2) +2n\ln(y) +O(n^{-1})\right] (1+O(n/m)) }{\exp\left( 2n\ln\left[ y\sqrt{n} +y^2n/2m +y^3n^{3/2}/6m^2 +O(n^2/m^3)\right] \right)}\\
 \times \frac{(e^{y\sqrt{n}/m} -1)m}{2n \sqrt{2\pi \left( \frac{y^2}{2} + \frac{y^4}{8n} +O(n^{-2})\right)}}\\
= \frac{\sqrt{ 2n} \exp\left[ n\ln(n)+y^2/16  +2n\ln(y) - 2n\ln(y) -n\ln(n) +O(n^{-1})\right] (1+O(n/m)) }{\exp\left( 2n\ln\left[ 1 +y\sqrt{n}/2m +y^2n/6m^2 +O(n^{3/2}/m^3)\right] \right)}\\
 \times \frac{   y\sqrt{n} +y^2 n/(2m) +O(n^{3/2}/m^2) }{2n \sqrt{  \frac{y^2}{2} + \frac{y^4}{8n} +O(n^{-2})}} \\
= \frac{\sqrt{ 2n} \exp\left[ y^2/16  +O(n^{-1})\right] }{\exp\left( 2n ( y\sqrt{n}/2m +y^2n/6m^2 - y^2n/4m +O(n/m^2) ) \right)}
 \times \frac{   y\sqrt{n} +y^2 n/(2m) +O(n^{3/2}/m^2) }{\sqrt{2}n y \sqrt{  1 + \frac{y^2}{4n} +O(n^{-2})}}  (1+O(n/m))\\ 
=\exp(y^2/16) (1+O(n^{-1})) \frac{ 1+O(\sqrt{n}/m)}{(1+ O(n\sqrt{n}/m) )(1+O(n^{-1}))} = \exp\left(\frac{y^2}{16}\right) \left(1 + O(n^{3/2}/m)\right). 
\end{multline}

We need to compute the normalized symplectic character for $\tau^0$, or else the normalized Schur function for $\nu^0=((\frac{1}{2})^{2n})$. Let $z+\frac{1}{2}=nw$, then we have
\begin{multline*}
S_{\nu^0}(e^{y\sqrt{n}/m};2n) 
= \frac{(2n-1)!}{(e^{y\sqrt{n}/m}-1)^{2n-1} }\frac{1}{2\pi \ii} \int \frac{e^{zy\sqrt{n}/m}}{\prod_{i=1}^{2n} (z - (\frac12+2n-i))  }dz\\
=  \frac{(2n-1)!n}{(e^{y\sqrt{n}/m}-1)^{2n-1} }\frac{1}{2\pi \ii} \int \frac{e^{(-\frac{1}{2}+nw)y\sqrt{n}/m}}{\prod_{i=1}^{2n} (nw-2n+i)  }dw\\
=\frac{(2n-1)!n}{(e^{y\sqrt{n}/m}-1)^{2n-1} }\frac{1}{2\pi \ii} \int \frac{e^{(-\frac{1}{2}+nw)y\sqrt{n}/m}}{\Gamma(nw)/\Gamma(nw-2n)  }dw\\
=\frac{(2n-1)!ne^{-\frac{y\sqrt{n}}{2m} }}{(e^{y\sqrt{n}/m}-1)^{2n-1} }\frac{1}{2\pi \ii} \int \frac{\exp\left( nwy\frac{\sqrt{n}}{m} - n(w)\ln(nw) +nw +n(w-2)\ln(nw-2n) -nw+2n\right)}{\sqrt{ (nw-2n)/nw }  } dw\\
=\frac{(2n-1)!ne^{-\frac{y\sqrt{n}}{2m} }}{(e^{y\sqrt{n}/m}-1)^{2n-1} }\frac{1}{2\pi \ii} \int \frac{\exp\left( -2n\ln(n)+2n + nwy\frac{\sqrt{n}}{m} - nw\ln(w) +n(w-2)\ln(w-2) \right)}{\sqrt{ (w-2)/w }  } dw
\end{multline*}
Let now 
$$\F_0(w,y) =wy \sqrt{n}/m -w\ln(w)+(w-2)\ln(w-2).$$ 
Let also $w_0$ be the positive root of $(\partial/\partial w)\F_0(w,y)=0$, so
$$y\sqrt{n}/m - \ln(w_0)+\ln(w_0-2) =0,$$
i.e. $w_0 = 2(1-e^{-y\sqrt{n}/m})^{-1}$, which is well beyond the poles. We also have that 
$$\F_0(w_0,y) = -2\ln(2) +2\ln(e^{y\sqrt{n}/m}-1)\;, \qquad \frac{\partial^2}{\partial w^2} \F_0(w,y)\Bigg|_{w=w_0} = \frac{(e^{y\sqrt{n}/m}-1)^2}{2e^{y\sqrt{n}/m}}.$$
Putting everything together, we obtain
\begin{multline*}
S_{\nu^0}(e^{y\sqrt{n}/m};n) = \frac{\sqrt{\pi n} e^{y\sqrt{n}/2m} \exp(-2n\ln(e^{y\sqrt{n}/m}-1)+2n\ln(2)+ n\F_0(w_0,y) )}{ (e^{y\sqrt{n}/m}-1)^{-1} \sqrt{\pi n (e^{y\sqrt{n}/m}-1)^2}}\\
= \exp(y\frac{\sqrt{n}}{2m} )(1+o(1)) =1+o(1)
\end{multline*}
Now  we can compute 
$$\X_{\tau^m}(e^{\frac{y\sqrt{n}}{m}};n) = \frac{2}{e^{y\sqrt{n}/m}+1} S_{\nu^m}(e^{y\sqrt{n}/m};2n) = \exp\left(  \frac{y^2}{16}\right)(1+o(1) )$$
$$\X_{\tau^0}(e^{y\sqrt{n}/m}) =1+o(1),$$
so that 
$$\Phi_m(e^{\frac{y\sqrt{n}}{m}};n) = \exp\left( y\sqrt{n}/2  +\frac{y^2}{16}\right)(1+o(1) )$$

Applying Proposition~\ref{prop:multivar} with $a_n=\frac{\sqrt{n}}{m}$ and $b_n = 0$ we have the multivariate versions of the above asymptotics, namely
\begin{align*}
\lim_{n\to \infty} \X_{\tau^m}( e^{\frac{y_1\sqrt{n}}{m}},\ldots, e^{\frac{y_k\sqrt{n}}{m}};n) &= \exp\left(  \frac{y_1^2}{16}+\cdots +\frac{y_k^2}{16}\right)\\
\lim_{n\to \infty} \X_{\tau^0}( e^{\frac{y_1\sqrt{n}}{m}},\ldots, e^{\frac{y_k\sqrt{n}}{m}};n) &= 1\\
\lim_{n\to \infty} \Phi_m( e^{\frac{y_1\sqrt{n}}{m}},\ldots, e^{\frac{y_k\sqrt{n}}{m}};n) \exp\bl - \frac{\sqrt{n}}{2}\sum_i y_i \br&= \exp\left(  \frac{y_1^2}{16}+\cdots +\frac{y_k^2}{16}\right).
\end{align*}
Finally, applying the same formulas and approach from Section~\ref{s:GUE} we see first that
$$B_k\bl x;\frac{y^k +\delta_k - \frac{m}{2} }{m/\sqrt{n}  }\br=  \exp\bl -\frac{\sqrt{n}}{2}\sum_i x_i \br B_k\bl x\sqrt{n}/m; y^k+\delta_k \br$$
and so using the asymptotics above we get
$$\mathbb{E}B_k(x;\frac{Y^k_{n,m}-m/2}{m/\sqrt{n} }) =  \exp\bl -\frac{\sqrt{n}}{2} \sum x_i \br \Phi_m\bl  e^{x_1 \sqrt{n}/m},\ldots, e^{x_k \sqrt{n}/m}; n\br \to \exp\left(  \frac{x_1^2}{16}+\cdots +\frac{x_k^2}{16}\right) .$$
Hence, as before, the so normalized $Y^k_{n,m}$ converges to $\GUE_k$. By the Gibbs property again, the shifted rescaled collection $\{Y^1_{n,m},\ldots,Y^k_{n,m}\}$ convergence to the GUE corners process. 
\end{proof}

\begin{proposition}\label{p:n_large}
Let $n,m\to \infty$, such that $n/m \to \infty$. Then 
$$\frac{ Y^k_{n,m}-m/2}{2\sqrt{m}} \to \GUE_k$$
and the collection of shifted rescaled positions of the horizontal lozenges on lines $1,\ldots,k$ converges to the $k\times k$ GUE-corners process. 
\end{proposition}

\begin{proof}
Let $a$ denote the factor $\frac{1}{\sqrt{m}}$, so that $ma^2=1$ We use the same approach and setup as in the proof of Proposition~\ref{p:n_small}, so we have 
$$S_{\nu^m}(e^{ay};2n) =  \frac{(2n-1)! e^{-ay/2}}{(e^{ay}-1)^{2n-1}} \frac{1}{2\pi \ii} \oint \frac{ \exp(uya+2n-f(u,m,n))}{\sqrt{\frac{(u-m/2-2n)(u+m/2-n)}{(u-m/2-n)(u+m/2)}} }(1 + O(1/m))du,$$
where $f(u,m,n)$ is defined there as well. Denote by $c=m/(2n)$ and remember that it is $o(1)$. Let also $u=nw+n$. After taking $\ln(n)$ and $n$ factors out we can rewrite $f(u,m,n)$  and $\F$ as
\begin{align*}f(u,m,n) = 2n\ln(n) + n &\left[ (w-c)\ln(w-c) + (w+c+1)\ln(w+c+1) \right. \\ &\left. - (w-c-1)\ln(w-c-1) - (w+c)\ln(w+c)\right],\\
\F(w;y) = a(w+1)y -&\left[ (w-c)\ln(w-c) + (w+c+1)\ln(w+c+1) \right. \\ &\left. - (w-c-1)\ln(w-c-1) - (w+c)\ln(w+c)\right]\end{align*}
We then find the positive (when $y>0$) root of 
$$\F'(w;y)= ay - \ln\left[ \frac{(w-c)(w+c+1)}{(w+c)(w-c-1)} \right]= 0$$
as
$$w_0 = \frac{e^{ay}+1}{e^{ay}-1}\left( 1 + (c^2+c) \bl \frac{e^{ay}-1}{e^{ay}+1}\br^2 \right) + O(c^2a^3) >c+1$$
for sufficiently large $m,n$ (i.e. small $c$ and $a$), so it is to the right of the integrand's poles. 
Further,
$$\F''(w_0;y) = \frac{-2c}{w_0^2-c^2} + \frac{2(c+1)}{w_0^2-(c+1)^2} = \frac{(e^{ay}-1)^2}{16} (1+ o(1)),$$
again since $a$ and $c$ are $o(1)$. Last, we have 
\begin{align*}
\F(w_0;y) = ay + yw_0 - w_0y + c \ln\left[ \frac{w_0^2-c^2}{w_0^2-(c+1)^2}\right] -\ln\left((w_0+c+1)(w_0-c-1)\right)\\
=ay +c \ln\left( \frac{ 1 - c^2\bl \frac{a^2y^2}{4} \br + O(c^2a^4) }{1 - (c+1)^2\bl \frac{a^2y^2}{4} \br +O(a^4) }\right) -\ln\bl \frac{4e^{ay}}{(e^{ay}-1)^2} +c^2 +O(c^2a^2) \br\\
=ay +c \bl - c^2\bl \frac{a^2y^2}{4} \br +  (c+1)^2\bl \frac{a^2y^2}{4}\br + O(a^4) \br   + 2\ln(e^{ay}-1) -2 -ay - c^2a^2y^2/4 + O(c^2a^3) \\
= ca^2y^2/4  +c^2 a^2y^2/4 + 2\ln(e^{ay}-1)  -2 + O(ca^3)
\end{align*}
Putting everything together we get
\begin{multline*}
S_{\nu^m}(e^{ay};2n) = \exp\bl 2n\ln(2n)-2n +2n-2n\ln(n) -ay/2 -2n\ln(e^{ay}-1) + nay  + nca^2y^2/4 \right. \\
\left. + 2n\ln(e^{ay}-1) -2n + nc^2a^2y^2/4 + O(nca^3) \br  (1 + O(1/\sqrt{m}))\\
=\exp(nay + nca^2 y^2/4 + nc^2a^2y^2/4) (1 +O(1/\sqrt{m}) )
\end{multline*}
Further, it can be seen in the proof of Proposition~\ref{p:n_small} that the analysis of $S_{\nu^0}$ does not depend on the scaling $a$ and so 
$$S_{\nu^0}(e^{ay};2n) = 1 +o(1).$$
So we have that 
$$\X_{\tau^m}(e^{ay};n) = \exp( nay + y^2/8 +o(1) ) \qquad \X_{\tau^0} (e^{ay};n)=1+o(1),$$
and the multiplicativity translates through Proposition~\ref{prop:multivar} to the same asymptotics for the multivariate $\X$. This finally gives
$$\Phi_m(e^{ay_1},\ldots,e^{ay_k};n) = \exp\bl (ma/2)\sum y_i +\frac{1}{8} \sum y_i^2 \br(1+o(1)).$$
Finally, performing the usual GUE analysis, we have
$$\mathbb{E}B_k\bl x; \frac{Y^k_{n,m}-m/2}{2\sqrt{m}}\br =  \exp\bl - \sqrt{m}\sum x_i\br \Phi_m(e^{2ax_1},\ldots,e^{2ax_k};n)\to \exp\left[ \frac{1}{2} (x_1^2+\cdots+x_k^2) \right]$$
so the shifted rescaled $Y^k_{n,m}$ converges to $\GUE_k$ and the collection $Y^1,\ldots,Y^k$ converges to the GUE-corners process.
\end{proof}

\section*{Acknowledgements}
We would like to thank Philippe Di Francesco for asking the author about GUE in the free boundary tilings.
We also thank Marek Biskup, Alexei Borodin, Vadim Gorin, Christian Krattenthaler, Oren Louidor and Nicolai Reshetikhin for helpful discussions on further problems and the background literature. 
The author was partially supported by a Simons postdoctoral fellowship while at UCLA.

\end{document}